\documentclass[11pt,a4paper]{amsart}
\usepackage{amsfonts,amsmath,amssymb,amsthm}
\usepackage{times}
\usepackage{bbm}
\usepackage{color}

\numberwithin{equation}{section}

\renewcommand{\epsilon}{\varepsilon}

\DeclareSymbolFont{SY}{U}{psy}{m}{n}
\DeclareMathSymbol{\emptyset}{\mathord}{SY}{'306}

\renewcommand{\div}{\mathrm{div}\,}

\DeclareMathOperator*{\esssup}{ess\,sup}

\DeclareMathSymbol{\newtimes}{\mathbin}{SY}{'264}

\newcommand{\R}{\mathbb{R}}

\newcommand{\uu}{\mathbf{u}}
\newcommand{\dd}{\mathbf{d}}
\newcommand{\ww}{\mathbf{w}}
\newcommand{\vv}{\mathbf{v}}
\newcommand{\nn}{\mathbf{n}}

{\bf}{\it}
{\bf}{\it}
{\bf}{\it}
{\bf}{\it}

{\bf}{\it}

{\bf}{\it}

\newtheorem{theorem}{Theorem}[section]{\bf}{\it}
{\bf}{\it}
\newtheorem{corollary}[theorem]{Corollary}{\bf}{\it}
{\it}{\rm}
\newtheorem{lemma}[theorem]{Lemma}{\bf}{\it}
{\it}{\rm}
\newtheorem{definition}[theorem]{Definition}{\bf}{\it}
{\bf}{\it}
{\bf}{\it}
{\bf}{\it}
{\bf}{\it}

\title[compactness of artificial compressibility approximations]{On the compactness of artificial compressibility approximations of weak solutions for fluid problems in deforming domains}

\author[S. Schmitz]{Stephan Schmitz}
\address{S.~ Schmitz, RPTU University Kaiserslautern-Landau, Department of Natural and Environmental Sciences, Institute of Mathematics,\newline Fortstrasse 7, D-76829 Landau, Germany.}
\email{stephan.schmitz@rptu.de}

\author[A. Hundertmark]{Anna Hundertmark}
\address{A.~ Hundertmark, RPTU University Kaiserslautern-Landau, Department of Natural and Environmental Sciences, Institute of  Mathematics,\newline Fortstrasse 7, D-76829 Landau, Germany.}
\email{a.hundertmark@rptu.de}

\subjclass[2010]{35D30, 35Q30, 74F10, 76D05, 76D03}

\keywords{Navier-Stokes-Equations, artificial compressibility approximation, deforming domains, weak solutions, compactness}

\copyrightinfo{2026}{A.~Hundertmark, 
S.~Schmitz}

\begin{document}
\begin{abstract}
 In this contribution, a fluid flow problem on a general deforming domain for a Newtonian fluid  in two and three space dimensions with artificial compressibility approximation is studied. We prove an estimate on the integral equicontinuity in time of the weak solutions under suitable domain regularity assumptions,  which is independent on the compressibility parameter and serves as an alternative compactness argument for the  convergence of weak solution sequences for vanishing compressibility.

 The corresponding estimate is obtained by remapping the problem onto a fixed  reference domain and using appropriate divergence-preserving testfunctions
involving the difference of two solutions at different points in time, thus defined with respect to different domains/coordinates. 

\end{abstract}

\maketitle

\section{Introduction and Main Result}\label{sec:intro}
 In recent years, much attention has come to Fluid-Solid Interaction (FSI) Problems in various scenarios  and applications, 
 from existence, uniqueness  and continuous dependence of (weak) solutions on data to their numerical  treatment, 
 see e.g. \cite{Canic}, \cite{Eberlein},   \cite{Grandmont}, \cite{Filo1}, \cite{anna1},  \cite{Padula},   \cite{Ruzicka} and the collections \cite{Galdi08}, \cite{Galdi14}  to name just a few publications in this area. 

In many cases, the  
 (weak) solvability of the FSI-problem is achieved by 
approximating the problem by a better treatable approximation, 
and verifying, that the corresponding weak  solutions indeed converge 
to solutions of the original problem.   Such approximating sequences can for instance be generated via parameters in the 
coupling fluid-structure conditions (see the authors previous work  \cite{Filo1}, \cite{anna2}), or parameters in the displacement equation (see e.g. \cite{Grandmont}, \cite{Grandmont08}, \cite{Padula}), a time discretization (see e.g. \cite{Canic}, \cite{anna2}), or via regularization operators and smoothing  or additional viscosity in the deformation equation  (see e.g. \cite{Eberlein}, \cite{Grandmont}, \cite{Ruzicka}). 

In this contribution,  we consider  artificial compressibility approximations for a Newtonian fluid-flow  (Navier-Stokes) problem  in time-dependent domains 
in space dimension two and three. The treatment of  non-linear convective terms in the momentum equation requires the strong convergence of the weak solution sequence in appropriate spaces. We present a detailed argumentation for strong convergence  of a sequence of such weak solution approximations in the appropriate $L^p$-space. The presented result can be applied to fluid-structure interaction problems, provided the necessary regularity of the domain deformation mapping is satisfied. 
  The crucial estimate  implying the compactness of the solution sequence is based  on the construction of appropriate  test functions, composed of a time-shift of the weak solution defined on the time-dependent, deforming domain. 
 A similar technique utilizing testfunctions that are additionally  solenoidal  w.r.t. different coordinates simultaneously has been used  
in the author's previous work \cite{anna2} on a 2D fluid-structure interaction problem.  In that result, in particular the continuous dependence  of the weak solution w.r.t.  the given domain deformation data is proved  finally resulting in the contractivity  of the fixpoint mapping  and the convergence of the corresponding geometry-iterative method.

  In what follows, we present and generalize this  technique for the proof  of compactness of artificial compressibility solution sequences 
  from a rectangular two dimensional reference domain \(D=(0,L)\times (0,1)\)  with a specific deformation function $\dd(y,t )=(0,h(y_1,t))^T$ describing the  vertical deformation of the upper part of the boundary 
  considered in  \cite{anna2},  to a general deforming domain \(\Omega_0\) in dimension two and three. We apply this technique to obtain an  \(L^2\)-bound of the time shift of the solution, 
 which we call the equicontinuity in time,  to finally obtain  the strong convergence of the series of approximations  in appropriate spaces and show the convergence of the compressible weak solution sequence to a divergence-free limit that solves the original weak problem.

 In this work the existence of an  artificially compressible weak solution \((\vv_\varepsilon,p_\varepsilon)\) and  the motion of the solid by its deformation \(\dd\) is assumed to be given.

Main result of this work is the following: 
\begin{theorem}\label{thm:Main1}
 Let \(\vv_\varepsilon \in  L^2
(0,T; H^1(\Omega_t))\cap L^\infty(0,T; L^2(\Omega_t))\) be a weak solution to the Navier-Stokes-Equations in the sense of Definition \ref{def:weak} in a time dependent domain \(\Omega_t\subset \mathbb{R}^n\) for \(n=2\) or \(n=3\),
\begin{equation*}\partial_t \vv_\varepsilon +(\vv_\varepsilon\cdot \nabla )\vv_\varepsilon-2\frac{\mu}{\rho}\div[e(\vv_\varepsilon)]+ \frac{1}{\rho}\nabla p_\varepsilon+\frac{1}{2} \vv_\varepsilon \div \vv_\varepsilon=\vec 0\text{ in  }\Omega_t\end{equation*} with additional artificial compressibility   term \(\varepsilon/\rho \Delta p_\varepsilon\) in the continuity
equation 
\begin{equation*}-\frac{\varepsilon}{\rho}\Delta p_\varepsilon +\div \vv_\varepsilon=0\text{ in }\Omega_t,\end{equation*} where 
\(\Omega_t=A_t(\Omega_0)\) with \(\Omega_0\) the reference domain at \(t=0\), \(A_t\colon \overline{\Omega_0}\to \overline{\Omega_t}\) the 
mapping between \(\Omega_0\) and \(\Omega_t\), \(e(\vv_\varepsilon)=\frac1{2}(\nabla \vv_\varepsilon +(\nabla \vv_\varepsilon)^T)\) is the symmetric deformation gradient.  The boundary  and coupling conditions for the velocity are valid as presented in Section \ref{initcons}. 
For the  artificial compressibility equation, the Neumann boundary conditions 
\(\nabla p_\varepsilon\cdot \vec{\nn}=0\text{ on }\partial \Omega_t\) with \(\vec{\nn}\) the outer unit normal at the boundary and the pressure normalization \(\int_{\Omega_t}p_\varepsilon dx=0\)  
 are applied.\\
In addition, in the case of \(n=3\), let \(\vv_\varepsilon \in L^4(0,T; L^4(\Omega_t))\) be uniformly bounded. 

\medskip

 Then the transformation \(\uu_\varepsilon\) of  \(\vv_\varepsilon\) to the fixed reference domain \(\Omega_0\), \(\uu_\varepsilon(y,t)=\vv_\varepsilon(A_t(y),t)=\vv_\varepsilon(x,t)\),  satisfies the following equicontinuity estimate uniformly in $\varepsilon$: 
\begin{equation}\label{eq:result}\int_0^{T-\tau}\int_{\Omega_0}\big|\sqrt{J}\uu_\varepsilon(y,s+\tau)-\sqrt{J}\uu_\varepsilon(y,s)\big|^2 dyds\leq C\tau,\end{equation} where \(J=\mathrm{det}(\nabla_yA_t)\) is the  Jacobian of the variable transformation between $\Omega_t$ and the reference domain $\Omega_0$.  
\end{theorem}
 The main ingredient in the proof of Theorem \ref{thm:Main1} is the construction of divergence preserving testfunctions at times \(s\) and 
\(s+\tau\) using the Piola transformation, compare the application of a similar technique in the works \cite{Padula}, \cite{anna2}. 
The equicontinuity estimate \eqref{eq:result} can be used to control the limiting process \(\varepsilon \to 0\) in the artificial compressibility approximation, and is presented in Section \ref{limit}. Shortly formulated, we get the following result: 

\begin{corollary}\label{thm:Main2}

Under the assumptions of Theorem \ref{thm:Main1},  the weak solution  sequence \(\vv_\varepsilon\) converges to the a limit \(\vv_0\), which is an  almost everywhere   divergence-free  weak solution to the incompressible Navier-Stokes-Equations
\begin{equation*}\partial_t \vv_0 +(\vv_0\cdot \nabla )\vv_0-2\frac{\mu}{\rho}\div[e(\vv_0)]+\nabla \frac{p_0}{\rho}=\vec 0,\quad \div( \vv_0)=0.\end{equation*} 
\end{corollary}

 The proof of Corollary \ref{thm:Main2} uses that 
by \eqref{eq:result} a subsequence of \(\uu_\varepsilon\) converges strongly  in \(L^p(0,T;L^p( \Omega_0)), 1\leq p<4\),  to the weak limit  \(\uu_0\in L^2(0,T;H^1(\Omega_0))\), which is divergence-free almost everywhere in $\Omega_0$. This argumentation is based on the compactness result by Alt-Luckhaus \cite[Lemma 1.9]{luckhaus}, and is presented in Section \ref{limit}. Similar argumentation for compactness  based on Simon's Lemma \cite{Simon} including time differences was also used in \cite{Grandmont08}, where the approximated sequences of  velocities come from an additional viscous term in the transversal displacement equation. The equicontinuity estimate is however obtained in a different way considering  an extension box domain to the deformed domain  without divergence preserving Piola transformation. 

For the limiting process $\varepsilon \to 0$  in Corollary \ref{thm:Main2}, the strong convergence of \(\uu_\varepsilon \to \uu_0\)  in appropriate spaces is necessary to eliminate the  artificial compressibility in the nonlinear Navier-Stokes problem in deforming domains. An equicontinuity estimate similar to   \eqref{eq:result} was at first presented by Alt-Luckhaus as a sufficient condition for compactness. Our result \eqref{eq:result} presents an alternative compactness argument for solution sequences on deforming domains. 

The classical Aubin compactness argument instead requires uniform a-priori estimates for \(\partial_t\uu_\varepsilon\) that are not {directly} available  in this setting. The compactness criterion by Lions requires the control of fractional derivatives via the Fourier transformation, which also  seems  not  to be straightforward for moving or deformable domains.
 For a discussion on  those compactness arguments, see \cite{Simon}, \cite[Chapter III]{Temam}.  We  also refer to  \cite{Muha} for a variant of the Aubins-Lions-Simon lemma on moving/deforming domains based on Simon's lemma \cite{Simon} with an Arbitrary Lagrangian-Eulerian (ALE)-type time-derivative and uniform integral equicontinuity.

The paper is organized as follows: In Section \ref{Formulation}, the precise formulation of the weak problem is stated. A preliminary investigation of the individual terms of the weak formulation  is given in Section \ref{sec:embedd}. In Section \ref{sec:first a-proiri}, the a-priori estimates for the weak solution  confirm the choice of functional spaces  for weak solutions. Section \ref{sec:second} is devoted to the proof of Theorem \ref{thm:Main1} by  testing the weak formulation with a special set of testfunctions. Based on Theorem \ref{thm:Main1}, the limiting process in  Theorem \ref{thm:Main2} is analyzed in Section \ref{limit} and our results are concluded in Section \ref{sec:conclusion}. Some technical issues, particularly on integral transformations, are for convenience addressed in the Appendix.

\section{Formulation of the problem}\label{Formulation}

\subsection{The domain}

We consider a  bounded deforming domain \(\Omega_t \subset \mathbb{R}^n, n=2, 3\),    that may be subject to in-/outflow of a fluid at an open (pervious) rigid boundary part \(\Gamma_{\mathrm{in}/\mathrm{out}}:=\Gamma_{\mathrm{in}}\cup\Gamma_{\mathrm{out}}\), a rigid non-pervious  non-slip boundary part \(\Gamma_{\mathrm{wall}}\) with positive surface measure  and a boundary part of fluid-solid interface \(\Gamma_{\mathrm{FSI}}\),  also with positive surface measure, attached to a deformable solid, that is \(\partial \Omega_t=\Gamma_{\mathrm{in}/\mathrm{out}}\cup \Gamma_{\mathrm{wall}}\cup \Gamma_{\mathrm{FSI}}\). The FSI-boundary part is subject to the deformation of the attached solid. 

To describe the fluid domain deformation, we consider the domain \(\Omega_t=A_t(\Omega_0)\) as the deformation of a fixed bounded Lipschitz reference domain \(\Omega_0=\Omega_t|_{t=0}\) by an ALE-map\\ \(A_t\colon \overline{\Omega_0}\to \overline{\Omega_t}\), which is assumed in this work to be given. The analysis presented is related to the so-called one-way coupling for fluid-solid-interaction problems and we are not treating the fully coupled FSI-problem. 
We distinguish the coordinates \(x=x(t)\) in the deforming domain  \(\Omega_t\) from the coordinates \(y\) in the fixed reference domain \(\Omega_0\) and use the relation \[x(t)=A_t(y)=y+\dd(y,t),\] where \(\dd\) is the given  additive deformation of the reference domain \(\Omega_0\) with \(\dd|_{\Gamma_{\mathrm{in}/\mathrm{out}}\cup \Gamma_{\mathrm{wall}}}=\vec{0}\). In particular, the rigid  \(\Gamma_{\mathrm{in}/\mathrm{out}}\) and \(\Gamma_{\mathrm{wall}}\) are invariant under \(A_t\) in the sense of the ALE-methodology. We suppose that  \(\dd=\dd(y,t)\)  with \begin{equation}\label{ass:d} \dd\in C^2(0,T;C^2(\overline{\Omega_0})) \end{equation} for some fixed time \(T>0\). In terms of Sobolev spaces, this smoothness can be granted by assuming \(\dd\in H^3(0,T;W^{3,n+1}(\Omega_0))\) (see \cite[Theorem 4.12]{Adams}). 

In addition, we assume that 
\begin{equation}\label{ass:globdiff}A_t\colon \overline{\Omega_0} \to \overline{\Omega_t} \text{ is a global \(C^2\)-diffeomorphism}\end{equation} 
and make the reasonable assumption (see Lemma \ref{invertible} below) that 
\begin{equation}\label{boundd2}|\partial_{y_i}\dd^j(y,t)|\leq K< \frac{1}{n}.\end{equation} 
For the description of the transformation we introduce \(\ww\) as the domain deformation-velocity \begin{equation}\label{defW}\ww:=\frac{\partial \dd}{\partial t}(y,t)\in  C^1(0,T;C^2(\overline{\Omega_0})).\end{equation}  
Since the reference domain \(\Omega_0\) is assumed to be Lipschitz and \(A_t\) is sufficiently smooth, \(\Omega_t\) is a Lipschitz domain as well.

For the description of the transformed terms we introduce the Jacobi matrix of \(A_t\)
\[\mathbb{J}=\mathbb{J}(y,t):=I+\left(\frac{\partial \dd^i}{\partial y_j}(y,t)\right)_{i,j=1,\dots,n}=:(\nabla_yA_t(y))^T\] as an (in general) non-symmetric \(n\times n\) matrix, where \(\nabla \vec{f}\) is understood as \((\partial_if_j)_{ij}\). We then set \[J=J(y,t):=\mathrm{det}(\mathbb{J})\] 
and introduce the well defined matrix (see Lemma \ref{invertible} below) \begin{equation}\label{defR}\mathbb{R}:=J\mathbb{J}^{-1}\end{equation}
with \(\det(\mathbb{R})=J^{n-1}\).

\subsection{Fluid motion}
In \(\Omega_t\) the fluid  moves with a fluid-velocity \(\vv_\varepsilon(x,t)=\vv_\varepsilon(x(t),t)\). 
Physically, we consider the motion of the fluid governed by the  Navier-Stokes equations for artificially compressible fluid with a constant density \(\rho\)
\begin{equation}\label{first equation}\partial_t \vv_\varepsilon +(\vv_\varepsilon\cdot\nabla )\vv_\varepsilon-2\frac{\mu}{\rho}\div[e(\vv_\varepsilon)]+\frac{1}{\rho}\nabla p_\varepsilon+\frac{1}{2}\vv_\varepsilon \div \vv_\varepsilon=\vec 0 \text{ in }\Omega_t,
\end{equation}
\begin{equation}\label{presure eq}
-\frac{\varepsilon}{\rho}\Delta p_\varepsilon +\div \vv_\varepsilon=0\text{ in }\Omega_t,\end{equation}
in the deforming domain \(\Omega_t\) (cf. \cite[Section III.8.1.1]{Temam}) with \(e(\vv_\varepsilon)=\frac{1}{2}(\nabla \vv_\varepsilon +(\nabla \vv_\varepsilon)^T)\) the symmetric deformation tensor. 

For \(\varepsilon >0\) in general  \(\div \vv_\varepsilon \neq 0\), which in turn implies 
a kind of artificial compressibility, whereas the fluid is incompressible in the limit \(\varepsilon \to 0\). 
An additional term \(\frac{1}{2} \vv_\varepsilon \div \vv_\varepsilon\) is included in \eqref{first equation} to balance the convective term \((\vv_\varepsilon \cdot \nabla) \vv_\varepsilon\) for the loss of the  solenoidal property as explained in Section \ref{bhat} below. Note that the term \(\frac{1}{2} \vv_\varepsilon \div \vv_\varepsilon\) vanishes for \(\varepsilon \to 0\) in the artificial compressibility formulation. This term will later be treated together with the nonlinear convective term \((\vv_\varepsilon \cdot \nabla)\vv_\varepsilon\). 

The equation \eqref{presure eq} is complemented by boundary and integrability conditions for the pressure 
\begin{equation}\label{eq:Normalization} \nabla p_\varepsilon\cdot \vec{\nn}=0\text{ on }\partial \Omega_t, \qquad\int_{\Omega_t}p_\varepsilon dx=0.\end{equation}
The velocity and pressure are  transformed  to the fixed reference domain \(\Omega_0\) as
\[\uu_\varepsilon(y,t):=\vv_\varepsilon(A_t(y),t)=\vv_\varepsilon(x(t),t), \quad q_\varepsilon(y,t):=p_\varepsilon(A_t(y),t)=p_\varepsilon(x(t),t).\]
In the following  sections, we will suppress that the velocity \(\vv=\vv_\varepsilon\) and the pressure \(p=p_\varepsilon\) depend on \(\varepsilon\) for easier readability. 

\subsection{Initial and boundary conditions}
\label{initcons}

The boundary conditions for the momentum equation are  considered as follows. 
We introduce  the  reference fluid-solid interface \[\Gamma_{\mathrm{FSI}}^0:=(A_t)^{-1}(\Gamma_{\mathrm{FSI}})\] as the back-transformation of the deformed part of the boundary to the reference domain. At the FSI-interface the kinematic coupling condition is imposed,
\begin{equation}\label{eq:w}\vv_\varepsilon|_{\Gamma_{\mathrm{FSI}}}=\uu_\varepsilon|_{\Gamma_{\mathrm{FSI}}^0}=\ww|_{\Gamma_{\mathrm{FSI}}^0}=\frac{\partial \dd}{\partial t}|_{\Gamma_{\mathrm{FSI}}^0}.\end{equation}

Moreover, we assume that the stress tensor \(\sigma_s\) of the solid acting on \(\Gamma_{\mathrm{FSI}}\)
 is given,\\ \(\sigma_s\in L^2(0,T;L^2(\Gamma_{\mathrm{FSI}}^0))\). Following the principle of continuity of stresses, we suppose 
\begin{equation}\label{tensor}\sigma_s\vec{\nn}=\sigma_f^0\vec{\nn},\end{equation} where \(\sigma_f^0\) denotes the back-deformed Cauchy fluid stress tensor to the reference domain. 

The fluid Cauchy stress tensor in the deforming domain reads
\begin{equation}\label{tensorf}\sigma_f:=-\mu(\nabla \vv_\varepsilon+\nabla \vv_\varepsilon^T)+ pI=-2\mu e(\vv_\varepsilon)+pI\end{equation}
and the transformation  to the fixed reference domain boundary part \(\Gamma_{\mathrm{FSI}}^0\) is by \eqref{eq:gradtransform} and \eqref{eq:tensortransform}    
\begin{equation}\label{tensor0}\sigma_f^0\big|_{\Gamma_{\mathrm{FSI}}^0}:=-\mu J^{-1} (\mathbb{R}^{T}\nabla_y\uu_\varepsilon  +  (\mathbb{R}^{T}\nabla_y\uu_\varepsilon )^T)\mathbb{R}^T+q_\varepsilon\mathbb{R}^T.\end{equation}

At the open (pervious) boundaries \(\Gamma_{\mathrm{in/out}}\) we impose the Neumann type boundary condition including the dynamic pressure \begin{equation}\label{kinematicp}2\mu (e(\vv_\varepsilon)\vec{\nn})\cdot \vec{\nn}-p_\varepsilon \bigg|_{\Gamma_{\mathrm{in}/\mathrm{out}}}=-p_{\mathrm{in}/\mathrm{out}}+\frac{\rho}{2}(\vv_\varepsilon \cdot \vec{\nn})^2\bigg|_{\Gamma_{\mathrm{in}/\mathrm{out}}},\quad \vv_\varepsilon\times \vec{\nn}|_{\Gamma_{\mathrm{in}/\mathrm{out}}}=0\end{equation}
to regulate the in-/outflow by given  boundary pressures $p_{\mathrm{in/out}}\in L^2(0,T; L^2(\Gamma_{\mathrm{in/out}}))$.
  
 At the fixed rigid wall-boundary 
$\Gamma_{\mathrm{wall}}$, we impose the no-slip boundary condition\begin{equation}\label{vwall}\uu_\varepsilon\big|_{\Gamma_{\mathrm{wall}}}=0.\end{equation} 
Finally, we impose that the fluid is at rest at the beginning,  \(\uu_\varepsilon(x,t=0)=0\). 
\medskip

The weak formulation of \eqref{first equation} - \eqref{presure eq} considering the above boundary and initial conditions is rewritten to the reference domain, using transformation of the integrals, integration by parts and the chainrule for derivatives,  some of the calculations are presented in  \eqref{eq:inttransformation} - \eqref{eq:tensortransform} in the Appendix.   
The transformed formulation is scaled with \(1/\rho\) and stated in the following definition. 

\begin{definition}\label{def:weak}
 Assume that the domain $\Omega_0$ deforms with deformation and velocity satisfying \eqref{ass:d} - \eqref{defW}. Assume  boundary data $p_{\mathrm{in/out}}\in L^2(0,T; L^2(\Gamma_{\mathrm{in/out}})), \sigma_{s}\in L^2(0,T; L^2(\Gamma_{FSI}^0))$.

The pair \((\uu_\varepsilon,q_\varepsilon)\),  \begin{equation}\label{assu}\uu_\varepsilon\in  L^2
(0,T; H^1(\Omega_0))\cap L^\infty(0,T; L^2(\Omega_0)), \quad q_\varepsilon\in L^2(0,T;H^1(\Omega_0)), \end{equation} is a weak solution to the \(\varepsilon\)-approximate fluid-flow problem \eqref{first equation} -   \eqref{presure eq} transformed to the reference domain \(\Omega_0\) satisfying the boundary conditions \eqref{eq:w}, \eqref{tensor}, \eqref{kinematicp}, \eqref{vwall} and the zero initial condition if it fulfills the following identity
\begin{align}\label{weakall}
0&= \int_0^T\bigg[-\int_{\Omega_0} \uu_\varepsilon \cdot (\partial_t J \psi+J\partial_t \psi) dy-\int_{\Omega_0}(\nabla_y\uu_\varepsilon)^T\mathbb{R}\ww\cdot\psi dy\nonumber\\
&\qquad+ \int_{\Omega_0}(\nabla_y\uu_\varepsilon)^T \mathbb{R}\uu_\varepsilon \cdot \psi dy+ \frac{1}{2}\int_{\Omega_0}\div(\mathbb{R}\uu_\varepsilon)\uu_\varepsilon \cdot \psi dy 
 \\ &\qquad+\frac{\mu}{\rho}\int_{\Omega_0}J^{-1}[\mathbb{R}^T\nabla_y\uu_\varepsilon dy+(\mathbb{R}^T\nabla_y\uu_\varepsilon)^T]:[\mathbb{R}^T\nabla_y\psi]dy-\int_{\Omega_0}\frac{q_\varepsilon}{\rho}\div(\mathbb{R}\psi)dy\nonumber\\&\qquad +\frac{\varepsilon}{\rho^2}\int_{\Omega_0}J^{-1}(\mathbb{R}^{T}\nabla q_\varepsilon)\cdot(\mathbb{R}^{T}\nabla\Phi)  dy+\frac{1}{\rho}\int_{\Omega_0}\div (\mathbb{R}\uu_\varepsilon) \Phi dy \nonumber\\&\qquad{ +\frac{1}{\rho}\int_{\Gamma_{\mathrm{FSI}}^0}\sigma_s\vec{\nn}\cdot \psi dS(y) + \int_{\Gamma_{\mathrm{in}/\mathrm{out}}}\frac{p_{\mathrm{in}/\mathrm{out}}}{\rho}\mathbb{R}^T\vec{\nn}\cdot \psi-\int_{\Gamma_{\mathrm{in}/\mathrm{out}}}\frac{(\uu_\varepsilon \cdot\psi)}{2}\mathbb{R}\uu _\varepsilon\cdot \vec{\nn}\bigg]dt}\nonumber\\ & \qquad +\int_{\Omega_0} J(T)\uu_\varepsilon(T)\cdot\psi(T)dy\nonumber
 \end{align}
for velocity testfunctions \(\psi\)  and  pressure testfunctions \(\Phi\) with
\begin{equation}\label{def:psi}\psi\in H^1(0,T;H^1(\Omega_0))
, \quad \psi|_{\Gamma_{\mathrm{wall}}}=0,\quad \Phi\in L^2(0,T;H^1(\Omega_0)).\end{equation} 
\end{definition}
 
Note that by assumption \(\uu_\varepsilon\times \vec{\nn}|_{\Gamma_{\mathrm{in}/\mathrm{out}}}=0\) the velocity \(\uu_\varepsilon\) is parallel to the exterior unit-normal \(\vec{\nn}\) at this boundary part, so that \(\vec{\nn}=\pm \frac{\uu_\varepsilon}{|\uu_\varepsilon|}\) with \(+\) sign at the outlet and \(-\) sign at the inflow part.  
Thus the corresponding boundary integral coming from the dynamic pressure condition \eqref{kinematicp} was replaced in \eqref{weakall}  by
\[\int_{\Gamma_{\mathrm{in}/\mathrm{out}}} (\uu_\varepsilon \cdot \vec{\nn})^2\mathbb{R}^T\vec{\nn}\cdot \psi dS(y)= \int_{\Gamma_{\mathrm{in}/\mathrm{out}}}(\uu_\varepsilon \cdot \psi)\mathbb{R}\uu_\varepsilon \cdot \vec{\nn} dS(y).\]
Moreover, the dynamic coupling condition of stress continuity  \eqref{tensor} is already applied in \eqref{weakall} resulting in a boundary integral containing \(\sigma_s\), while the kinetic coupling condition \eqref{eq:w} will be used  later,  cf. \eqref{eq:later} and \eqref{est:X}.

\section{Preliminaries} \label{sec:embedd}

In this notation the product \(a\cdot b:=a^Tb\) denotes the scalarproduct of vectors, the product \(A:B:=\mathrm{trace}(A^TB)\) the corresponding tensorproduct. The gradient of a scalar function \(\nabla f\) is a column-vector.  Correspondingly, for vector-valued \(\vec{f}=(f_1,\dots, f_n)^T\), we set \(\nabla \vec{f}:=(\partial_i f_j)_{i,j=1,\dots,n}\), the transposed Jacobian matrix. The divergence of a matrix is taken column-wise.

We further denote the \(\mathbb{R}^n\)-vector Euclid  norms by \(|\cdot|\) and  the standard  \(L^p\)- norms by\\ \(\|\uu\|_{L^p(\Omega_0)}:=(\int_{\Omega_0}|\uu (y)|^pdy)^{1/p}\) for \(p<\infty\) and \(\|\uu\|_{L^\infty(\Omega_0)}=\esssup_{y\in \Omega_0} |\uu(y)|\).
Correspondingly, for a  matrix-valued  function \(M\) we use the induced Frobenius-norm \(|M|_F\).
Where no confusion can arise, \(L^p\)-norms on the same set as in the first term in the  corresponding expression are  just written as \(\|\cdot\|_{L^p}\).

For simplificated notation, since \(J^{n-1}=\det(\mathbb{R}), \mathbb{J}=J\mathbb{R}^{-1}\), dependence  of constants on \(J\) and \(|\mathbb{J}|_F\) will be indicated as dependence on \(\mathbb{R}\). The dependence on  \(\mathbb{R}, \mathbb{R}^{-1}\) and the derivatives \(\partial_t\mathbb{R},  \partial_t\mathbb{R}^{-1},\nabla\mathbb{R}, \nabla\mathbb{R}^{-1}\) is in fact  a dependence on certain derivatives of the deformation \(\dd\) and we will use the simplified notation \(C(\mathbb{R}, \mathbb{R}^{-1},\partial_t\mathbb{R},  \partial_t\mathbb{R}^{-1},\nabla\mathbb{R}, \nabla\mathbb{R}^{-1})=:C(\mathbb{R},\partial_t\mathbb{R},\nabla\mathbb{R})\).

\subsection{Regularity of the ALE-map \(A_t\) and of  \(\mathbb{J}\)}
\begin{lemma}\label{invertible}
Suppose that \eqref{boundd2} is fulfilled, that is  \(|\partial_{y_i}\dd^j|\leq K< \frac{1}{n},\) then
\begin{enumerate}
    \item[(a)] \(\mathbb{J}=(\nabla_y A_t)^T\) and its determinant \(J\) are boundedly invertible and thus \(\mathbb{R}\) is well defined, cf. \eqref{defR}.
    \item[(b)] The ALE-map \(A_t\colon \Omega_0 \to \Omega_t\) is a local diffeomorphism.
\end{enumerate}
\begin{proof}
\begin{itemize}
\item[(a)] 
By the classical Gershgorin Circle Theorem \cite{Gershgorin}, the following spectral inclusion for the complex eigenvalues of \(\mathbb{J}\) is valid for \(n=2\),   \begin{equation}\label{Gershgorin1}\mathrm{spec}(\mathbb{J}(y,t))\subseteq B_{|\frac{\partial \dd^1}{\partial y_2}|}\left(1+\frac{\partial \dd^1}{\partial y_1}\right) \cup B_{|\frac{\partial \dd^2}{\partial y_1} |}\left(1+\frac{\partial \dd^2}{\partial y_2}\right)=:S\subseteq B_{2K}(1).\end{equation}
Here \(B_R(p)\) denotes the ball with center \(p\) and radius \(R\). 
If  \(K:=\mathrm{sup}_{i,j}|\frac{\partial \dd^i}{\partial y_j}|<1/2\), then \(\mathrm{dist}(0, S)\geq 1-2K>0\), that is \(0\notin\mathrm{spec}(\mathbb{J}(y,t))\). Thus  all spectral values are away from \(0\)  and thus \(\mathbb{J}(y,t)\) is  boundedly invertible,  \(\mathbb{J}^{-1}=J^{-1}\mathrm{cof}(\mathbb{J}^T)\) is uniformly bounded for \(t\in (0,T)\) and  and \(J>\mathrm{const}>0\). The inclusion \eqref{Gershgorin1} can be generalized  to the case of \(n=3\) with \(K<1/3\) and
\begin{equation}\label{Gershgorin2}\begin{aligned}\mathrm{spec}(\mathbb{J}(y,t))&
 =:S\subseteq B_{3K}(1).\end{aligned}\end{equation}
\item[(b)] The local invertibility of \(A_t\) follows  directly from the invertibility of \(\mathbb{J}\) by the classical Theorem on local invertibility. 
\end{itemize}
\end{proof}
\end{lemma}
Note that by assumption \eqref{boundd2} and the above lemma above \(A_t\) is a local diffeomorphism. However for our investigations it is necessary that \(A_t\) is a global diffeomorphism, see \eqref{ass:globdiff}, which is not guaranteed by \eqref{boundd2} and a stronger and harder to proof assumption. 

\subsection{Sobolev and boundary Imbeddings}

Next, we collect the following useful imbeddings for bounded Lipschitz-domains depending on the dimension for the case \(n=2\) or \(n=3\):
\begin{enumerate}
\item[\textbf{n=2:}]
One has that 
\(H^1(\Omega_0)\hookrightarrow L^q(\Omega_0),\) for \( 2\leq q< \infty\) (see \cite[Theorem II.2.4]{Galdi}).\\
In particular, by the Gagliardo-Nirenberg interpolation-inequality \cite[Exercise II.2.9]{Galdi},
\begin{equation}\label{n=2,Omega0}\|u\|_{L^4(\Omega_0)}\leq C(\Omega_0)\|u\|_{L^2(\Omega_0)}^{1/2}\|u\|_{H^1(\Omega_0)}^{1/2}\leq C(\Omega_0)\|u\|_{H^1(\Omega_0)}\text{ for } u\in H^1(\Omega_0).\end{equation}
For the boundary imbeddings we have  \(H^1(\Omega_0)\hookrightarrow L^q(\partial \Omega_0)\), for  \(1\leq q<\infty\). In particular, by \cite[Theorem II.3.1]{Galdi},  the following  estimates hold
\begin{equation}\label{n=2,deltaOmega3}\|u\|_{L^3(\partial \Omega_0)}\leq C(\Omega_0)\|u\|_{L^2(\Omega_0)}^{1/2}  \; \|u\|_{H^1(\Omega_0)}^{1/2}\leq C(\Omega_0)\|u\|_{H^1(\Omega_0)},
\end{equation}
and also 
\begin{equation}\label{n=2,deltaOmega2}\|u\|_{L^2(\partial  \Omega_0)}\leq C(\Omega_0) \|u\|_{L^2(\Omega_0)}\leq C(\Omega_0)\|u\|_{H^1(\Omega_0)},\end{equation} where we suppressed the trace operator in the notation.
\item[\textbf{n=3:}]
The Sobolev-imbeddings for bounded Lipschitz-domains (see \cite[Exercise II.2.9]{Galdi}) imply \(H^1(\Omega_0)\hookrightarrow L^q(\Omega_0),\) for \( 2\leq q \leq 6\) and in particular
\begin{equation}\label{n=3,Omega0}\|u\|_{L^4(\Omega_0)}\leq C(\Omega_0)\|u\|_{L^2(\Omega_0)}^{1/4} \| u\|_{H^1(\Omega_0)}^{3/4}\leq C(\Omega_0)\|u\|_{H^1(\Omega_0)}\text{ for } u\in H^1(\Omega_0).\end{equation} 
  By the boundary imbedding \cite[Theorem II.3.1]{Galdi}, we have \begin{equation}\label{n=3,deltaOmega4}\|u\|_{L^4(\partial\Omega_0)}\leq C(\Omega_0) \|u\|_{H^1(\Omega_0)}.\end{equation}
Moreover,  the imbedding chain \(L^4(\partial \Omega_0)\hookrightarrow L^3(\partial \Omega_0)\hookrightarrow L^2(\partial \Omega_0)\)  holds  
by the standard  H\"older-inequality for bounded surface measure of \(\partial \Omega_0\). 

In the case \(n=3\), the corresponding estimate  for \(L^2(\partial \Omega_0)\) can however be refined to
\begin{equation}\label{n=3,deltaOmega2}
\|u\|_{L^2(\partial \Omega_0)}\leq C(\Omega_0) \|u\|_{L^2(\Omega_0)}^{1/2}  \; \|u\|_{H^1(\Omega_0)}^{1/2}.
\end{equation} 
This estimate can be obtained following the lines of the proof of \cite[Theorem 1.5.1.10]{Grisvard} in combination with the generalized Poincar\'e inequality \eqref{GenPoinc}.
\item[\textbf{n=1:}]
With respect to the time-variable, the following  imbedding chain is useful \begin{equation}\label{d=1}H^1((0,T))\hookrightarrow C([0,T])\hookrightarrow L^p([0,T]),\quad 1\leq p\leq \infty,\end{equation} see \cite[Theorem 4.12]{Adams}.
\end{enumerate}

Note that the above imbedding  estimates are also valid with \(L^2\)-norms of  the gradient \(\nabla u\) instead of \(H^1\)-norms by the generalized Poincar\'e inequality \eqref{GenPoinc} in the case of zero boundary values on a part of the boundary.

\subsection{Continuity and coercivity of bilinear forms}

To simplify the notation, we introduce the following forms on \(H^1(\Omega_0)\):
\begin{definition}
We introduce a bilinear form on \(H^1(\Omega_0)\) corresponding to the transformed viscous term by 
 \begin{equation}\label{def:viscosform}((\uu, \psi)):=\frac{\mu}{\rho}\int_{\Omega_0}J^{-1}[\mathbb{R}^T\nabla_y\uu+(\mathbb{R}^T\nabla_y\uu)^T]:[\mathbb{R}^T\nabla_y\psi]dy.\end{equation}
The form \(((\,\cdot\,,\,\cdot\,))\) is bounded, symmetric and coercive.
 \end{definition} 
 The symmetry \(((\uu,\psi))=((\psi, \uu))\) follows directly from the rules for traces of square matrices: \(\mathrm{trace}((A+A^T)^TB)=\mathrm{trace}(A^TB+AB)=\mathrm{trace}(B^TA+BA)=\mathrm{trace}((B+B^T)^TA)\). The  boundedness of \eqref{def:viscosform} follows from 
 \[\begin{aligned}|((\uu,\psi))|&\leq C\left\|J^{-1}\right\|_{L^\infty(\Omega_0)}\left\|\mathbb{R}\right\|^2_{L^\infty(\Omega_0)}\left\|\nabla \uu\right\|_{L^2(\Omega_0)} \left\|\nabla \psi\right\|_{L^2(\Omega_0)}\\ &\leq C_1(\mathbb{R}) \left\|\uu\right\|_{H^1(\Omega_0)}\left\|\psi\right\|_{H^1(\Omega_0)}.\end{aligned}\]
To show the coercivity of $(( \cdot, \cdot ))$,  let us denote the transformed analogon of the symmetric gradient \(e(\vv)\), see \eqref{first equation},
\[\hat{e}(\uu):=\frac{\mathbb{R}^T\nabla_y\uu+(\mathbb{R}^T\nabla_y\uu)^T}{2}.\] 
Due to the symmetry of the form
and the fact that  \(\det(\mathbb{R})=J^{n-1}>\mathrm{const} \geq0\) by assumption \eqref{boundd2} on the deformation \(\dd\),
we can apply the generalized Korn-inequality  \cite[Corollary 4.1]{pompe} and the generalized Poincar\'e inequality \eqref{GenPoinc}, \cite[Corollary 4.5.2]{Ziemer},  to obtain the coercivity estimate 
\begin{equation}\label{coerciv1}((\uu,\uu))=\frac{2\mu}{\rho}\int_{\Omega_0}J^{-1}\hat{e}(\uu):\hat{e}(\uu)dy\geq C(\mathbb{R})\int_{\Omega_0}|\nabla \uu|^2 \geq c_1(\Omega_0,\mathbb{R})\|\uu\|_{H^1(\Omega_0)}^2.\end{equation}  

 \begin{definition} We  introduce a bilinear  form on \(H^1(\Omega_0)\)  
\begin{equation}\label{def:a1} a(q, \Phi):=\int_{\Omega_0}J^{-1}(\mathbb{R}^{T}\nabla q)\cdot (\mathbb{R}^{T}\nabla\Phi)dy\end{equation}associated with the pressure Laplacian term. This form is bounded, symmetric and coercive.
\end{definition}
The symmetry is obvious and the boundedness follows easily since
 \[|a(q, \Phi)| \leq \left\|\mathbb{R}\right\|_{L^\infty(\Omega_0)}^2\cdot \left\|J^{-1}\right\|_{L^\infty(\Omega_0)} \left\|\nabla q\right\|_{L^2(\Omega_0)} \left\|\nabla \Phi\right\|_{L^2(\Omega_0)}.\]
To see the coercivity, note that \(J^{-1}\mathbb{R}\mathbb{R}^{T}\) is symmetric. Moreover, due to Lemma \ref{invertible}, \(\mathbb{J}\), \(\mathbb{J}^{-1}\) and correspondingly \(\mathbb{R}, \mathbb{R}^{-1}\) are boundedly invertible with  spectra isolated  from \(0\). Thus, the singular values of \(\mathbb{R}\) 
 then are strictly positive.
Using the standard estimate for the Rayleigh-quotient (see e.g. \cite[Lemma 23.5]{Hanke}), one has
\[\begin{aligned}&J^{-1}(\mathbb{R}^{T}\nabla q)\cdot(\mathbb{R}^{T}\nabla q)
\geq 
(\inf J^{-1})  \mathbb{R}\mathbb{R}^{T}\nabla q\cdot\nabla q\\
&\geq (\inf J^{-1}) \inf \mathrm{spec}(\mathbb{R}\mathbb{R}^{T})\nabla q\cdot \nabla q=C(\mathbb{R})\left|\nabla q\right|^2\text{ for some } C>0.\end{aligned}\]In particular \(a(q,q)\geq C(\mathbb{R})\left\|\nabla q\right\|_{L^2(\Omega_0)}^2\). 
Since the pressure-average is normalized to zero, see \eqref{eq:Normalization}, 
 one can use the classic Poincar\'e-Wirtinger-inequality to obtain the coercivity \begin{equation}\label{coerciv2}a(q,q)\geq c_2(\Omega_0,\mathbb{R})\left\|q\right\|_{H^1(\Omega_0)}^2.\end{equation}

\subsection{Nonlinear convective terms}\label{bhat}
\begin{definition}
We introduce the nonlinear convective form transformed into \(\Omega_0\),
\begin{equation}\label{def:b} \hat b(\uu,\uu,\psi):= \int_{\Omega_0} (\nabla_y\uu)^T\mathbb{R}\uu \cdot \psi dy+
\frac{1}{2}\int_{\Omega_0}\div_y(\mathbb{R}\uu)\uu \cdot \psi dy.\end{equation}  
\end{definition}
Note that the form \eqref{def:b} contains the additional nonlinear term balancing the loss of\\ solenoidality of the weak solution approximation cf. \eqref{first equation} and lines thereafter.\\
We first state that \(\int_0^T |\hat b(\uu,\uu,\psi)(t)| dt
\) exists,  for \(\uu\) given in \eqref{assu} and \(\psi\) satisfying \eqref{def:psi}, since \(\mathbb{R}, \nabla_y \mathbb{R} \) are bounded (cf. \eqref{defR}, \eqref{ass:d}, Lemma \ref{invertible}).  Indeed, the first term in \eqref{def:b} can be estimated by 
\[\int_0^T\int_{\Omega_0} |(\nabla_y\uu)^T\mathbb{R}\uu \cdot \psi |dydt\leq C_1(\mathbb{R})\int_0^T \|\uu(t)\|_{L^4(\Omega_0)}\cdot \|\nabla \uu (t)\|_{L^2} \cdot \|\psi(t)\|_{L^4}dt. \]

In dimension \(n=2\), using the imbeddings \eqref{n=2,Omega0}, \eqref{d=1}, we further obtain
\[\begin{aligned} &\int_0^T \|\uu(t)\|_{L^4(\Omega_0)}\cdot \|\nabla \uu (t)\|_{L^2(\Omega_0)} \cdot \|\psi(t)\|_{L^4(\Omega_0)}dt\\ & \leq C_2(\Omega_0)\int_0^T\|\uu(t)\|_{L^2(\Omega_0)}^{1/2} \cdot\| \uu(t)\|_{H^1(\Omega_0)}^{3/2}\cdot  \| \psi(t)\|_{H^1(\Omega_0)} dt\\ 
& \leq  C_2(\Omega_0)\|\uu\|_{L^\infty(0,T;L^2(\Omega_0))}^{1/2}   \| \uu\|_{L^2(0,T;H^1(\Omega))}^{3/2} \| \psi\|_{C(0,T;H^1(\Omega_0))}.
\end{aligned}\]
In dimension \(n=3\), based on \eqref{n=3,Omega0} and \eqref{d=1},  we have
\begin{equation}\label{bn=3}\begin{aligned} &\int_0^T \|\uu(t)\|_{L^4(\Omega_0)}\cdot \|\nabla \uu (t)\|_{L^2(\Omega_0)} \cdot \|\psi(t)\|_{L^4(\Omega_0)}dt\\ & \leq C_2(\Omega_0)\int_0^T\|\uu(t)\|_{L^2(\Omega_0)}^{1/4} \cdot\|\uu(t)\|_{H^1(\Omega_0)}^{3/4}\cdot \| \uu(t)\|_{H^1(\Omega_0)}\cdot \|\psi(t)\|_{H^1(\Omega_0)} dt\\ 
&\leq C_2(\Omega_0)\|\uu\|_{L^\infty(0,T;L^2(\Omega_0))}^{1/4}  \| \uu\|_{L^2(0,T;H^1(\Omega_0))}^{7/4} \cdot \| \psi\|_{C(0,T;H^1(\Omega_0))}. \\
\end{aligned}\end{equation} 
The second term  \(\frac{1}{2}\int_0^T\int_{\Omega_0}|\div_y(\mathbb{R}\uu)\uu \cdot \psi|dydt\) of \eqref{def:b} can be estimated in an analogus way with a constant depending on \(\Omega_0, \mathbb{R}, \nabla_y \mathbb{R}\).

Applying integration by parts similar as in   \cite[Equation (2.12)]{anna1}, \cite[Remark 4]{Grandmont}, one gets  
\begin{equation}\begin{aligned}
\label{bterm}
& \hat b(\uu,\uu,\psi)=\\
&\frac1{2}\int_{\Omega_0}(\nabla_y \uu)^T \mathbb{R}\uu \cdot \psi dy 
-\frac1{2}\int_{\Omega_0}(\nabla_y \psi)^T \mathbb{R}\uu \cdot \uu dy 
+\frac1{2}\int_{\partial \Omega_0}(\uu\cdot \psi)\mathbb{R}\uu\cdot \vec{\nn} dS(y).\end{aligned}\end{equation}
 This implies the essential property of the cancellation of the volume integrals
\begin{equation}\label{simple1}
\begin{aligned}\hat b(\uu,\uu, \uu)&=\frac{1}{2}\int_{\partial \Omega_0}|\uu|^2 \mathbb{R} \uu \cdot \vec{\nn}dS(y)\\ &=
\frac{1}{2}\int_{ \Gamma_\mathrm{in/out}}|\uu|^2 \mathbb{R} \uu \cdot \vec{\nn}dS(y)+ \frac{1}{2}\int_{\Gamma_{\mathrm{FSI}}^0}|\ww|^2 \mathbb{R} \ww \cdot \vec{\nn}dS(y).
\end{aligned}\end{equation} 

\section{First a-priori estimates:}\label{sec:first a-proiri}

The natural function spaces for the weak solution \((\uu, q)\) from Definition \ref{def:weak} are confirmed by the following  a-priori estimates. The solutions themselfs are inserted into \eqref{weakall}.
In this case  the weak formulation \eqref{weakall} simplifies  as the two terms involving  the divergence 
cancel each other \[-\int_{\Omega_0}\frac{q}{\rho}\div(\mathbb{R}\psi)dy+\frac{1}{\rho}\int_{\Omega_0}\div (\mathbb{R}\uu) \Phi dy=0\] for  \(\psi=\uu \) and \(\Phi=q\). Consequently, we obtain the following simplified weak formulation where we replaced \(\uu =\ww\) at \(\Gamma_{\mathrm{FSI}}^0\) and applied the notations for forms  \eqref{def:viscosform}, \eqref{def:a1}:
\begin{eqnarray}\label{eq:later}
&0&=\int_{\Omega_0} J(T)|\uu|^2(T)dy +\int_0^T\bigg[-\int_{\Omega_0} \uu \cdot (\partial_t J \uu+J\partial_t \uu)dy-\int_{\Omega_0}(\nabla_y\uu)^T\mathbb{R}\ww\cdot\uu dy \nonumber\\ 
&&+((\uu, \uu)) +\frac{\varepsilon}{\rho^2} a(q,q)+\hat{b}(\uu,\uu,\uu)\\
&&+\frac{1}{\rho}\int_{\Gamma_{\mathrm{FSI}}^0}\sigma_s\vec{\nn}\cdot \ww dS(y)+ \int_{\Gamma_{\mathrm{in/out}}}\left(\frac{-1}{2}|\uu|^2\mathbb{R}\uu \cdot\vec{\nn} + \frac{p_{\mathrm{in}/\mathrm{out}}}{\rho}\mathbb{R}^T\vec{\nn}\cdot \uu\right) dS(y)\bigg]dt.\nonumber
\end{eqnarray}
We now formally substitute \(\partial_t \uu \cdot \uu=\frac{1}{2}\partial_t(|\uu|^2)\) in the second integral. Using integration by parts in time in the second term, 
\[-\int_0^T\uu\cdot \partial _t J \uu dt=-\int_0^T (\partial_tJ)|\uu|^2dt=-J|\uu|^2\big|_0^T+\int_0^T J \partial_t (|\uu|^2)dt,\]
the first term in \eqref{eq:later} vanishes and \(\frac{1}{2}\int_0^T\int_{\Omega_0}J\partial_t|\uu|^2 dydt\) remains on the right hand side.

As next, we rewrite \((\nabla_y\uu)\uu=\frac{1}{2}\nabla_y(|\uu^2|)\) to get a new representation for the third integral
\((\nabla_y\uu)^T\mathbb{R}\ww\cdot\uu=\mathbb{R}\ww\cdot(\nabla_y\uu)\uu=\mathbb{R}\ww\cdot \frac{1}{2}\nabla_y(|\uu^2|)\). Moreover, integration by parts of this integral cancels the corresponding  boundary term  \(-\int_{\Gamma_{\mathrm{FSI}}^0}\frac{1}{2}|\ww|^2\mathbb{R}\ww\cdot \vec{\nn} dS(y)\) over \(\Gamma_{\mathrm{FSI}}^0\) arriving in \(\hat{b}(\uu,\uu,\uu)\), cf. \eqref{simple1}. The remaining boundary terms 
 \(-\int_{\Gamma_{\mathrm{in/out}}\cup \Gamma_{\mathrm{wall}}}\frac{1}{2}|\uu|^2\mathbb{R}\ww\cdot \vec{\nn} dS(y)\) 
over the non-deforming boundary parts  $\Gamma_{\mathrm{in/out}}$ and $\Gamma_{\mathrm{wall}}$ disappear due to \(\uu|_{\Gamma_{\mathrm{wall}}}=0\), \(\ww|_{\Gamma_{\mathrm{in}/\mathrm{out}}}=0\).  
 In total, \eqref{eq:later} simplifies to
\begin{equation}\label{Spezialfall u}\begin{aligned}
0&= \int_0^T\bigg[\int_{\Omega_0}  \frac{J}{2}\partial_t (|\uu|^2)dy+\frac{1}{2}\int_{\Omega_0}\div_y(\mathbb{R}\ww)|\uu|^2dy +((\uu,\uu))\\ &\qquad+\frac{\varepsilon}{\rho^2} a(q,q) +\frac{1}{\rho}\int_{\Gamma_{\mathrm{FSI}}^0}\sigma_s\vec{\nn}\cdot \ww dS(y)+ \int_{\Gamma_{\mathrm{in}/\mathrm{out}}}\frac{p_{\mathrm{in}/\mathrm{out}}}{\rho}\mathbb{R}^T\vec{\nn}\cdot \uu dS(y)\bigg].
 \end{aligned}
\end{equation}
By the definition of the material (ALE-type) time-derivative \(D_t^A\)  in the deforming domain, we have \[\frac{1}{2} \int_{\Omega_0} J\partial_t(|\uu|^2)dy=:\frac{1}{2} \int_{\Omega_t}D^A_t(|\vv|^2)dx.\]
By the ALE-Transport Theorem \eqref{ALE} and the  transformations of integrals \eqref{eq:inttransformation}, \eqref{eq:divtransform}, we get
\[\begin{aligned} \frac{1}{2}\int_{\Omega_t}D^A_t(|\vv|^2)dx&=\frac{1}{2}\frac{d}{dt}\left(\int_{\Omega_t}|\vv|^2dx\right)-\frac{1}{2}\int_{\Omega_t}|\vv|^2\div_x \ww dx\\
&= \frac{1}{2}\frac{d}{dt}\left(\int_{\Omega_0}J|\uu|^2dy\right )-\frac{1}{2}\int_{\Omega_0}|\uu|^2\div_y(\mathbb{R}\ww) dy.
\end{aligned}\]
Since  \(\Omega_0\) does not depend on the time \(t\), we have \(\frac{d}{dt}\left(\int_{\Omega_0}J|\uu|^2dy\right)=\int_{\Omega_0}\partial_t(J|\uu|^2)dy\) and consequently, the first two integrals in \eqref{Spezialfall u} can be replaced by \(\frac{1}{2} \int_0^T\int_{\Omega_0} \partial_t(J|\uu|^2)dydt\). Finally, \eqref{Spezialfall u} can be rewritten as
\begin{equation}\label{Spezialfall u3}\begin{aligned}
0&= \int_0^T\bigg[\frac{1}{2}\int_{\Omega_0}\partial_t(J|\uu|^2)dy+((\uu,\uu))+ \frac{\varepsilon}{\rho^2} a(q,q)\\ & \qquad+\int_{\Gamma_{\mathrm{in}/\mathrm{out}}} \frac{p_{\mathrm{in}/\mathrm{out}}}{\rho}\mathbb{R}^T\vec{\nn}\cdot \uu dS(y)\bigg] dt+\int_0^T\frac{1}{\rho}\int_{\Gamma_{\mathrm{FSI}}^0}\sigma_s\vec{\nn}\cdot \ww dS(y),
\end{aligned}\end{equation} where the last term is given by the solid-stress action. 
With the initial condition \(\uu(t=0)=0\), Fubini and the fundamental theorem of calculus on the first term on the right, we arrive at
\begin{equation*}\label{Zeitintegral}\begin{aligned}
 &\frac{1}{2}\int_{\Omega_0}(J|\uu|^2)(T)dy+\int_0^T\bigg[((\uu,\uu))+ \frac{\varepsilon}{\rho^2} a(q,q)\bigg] dt\\ &=-  \int_0^T\int_{\Gamma_{\mathrm{in}/\mathrm{out}}} \frac{p_{\mathrm{in}/\mathrm{out}}}{\rho}\mathbb{R}^T\vec{\nn}\cdot \uu dS(y) dt-\int_0^T\int_{\Gamma_{\mathrm{FSI}}^0}\frac{1}{\rho}\sigma_s\vec{\nn}\cdot \ww dS(y)dt.
\end{aligned}\end{equation*} 
We now estimate the term including the given boundary pressure data using the imbedding    \eqref{n=2,deltaOmega2}, \eqref{n=3,deltaOmega4}, the generalized Poincar\'e-inequality and the Young-inequality. To this end, let \(\delta>0\) be arbitrary, then we can estimate
\[\begin{aligned}
& \left|\int_{\Gamma_{\mathrm{in}/\mathrm{out}}} \frac{p_{\mathrm{in}/\mathrm{out}}}{\rho}\mathbb{R}^T\vec{\nn}\cdot \uu dS(y)\right|\leq \frac{\|\mathbb{R}\|_{L^\infty(\overline{ \Omega_0})}}{\rho} \int_{\Gamma_{\mathrm{in}/\mathrm{out}}}|p_{\mathrm{in}/\mathrm{out}}|\; |\uu|dS(y)\\
& \leq C_1(\mathbb{R}) \|p_{\mathrm{in}/\mathrm{out}}\|_{L^{2}(\Gamma_{\mathrm{in}/\mathrm{out}})}\; \|\uu\|_{L^2(\Gamma_{\mathrm{in}/\mathrm{out}})}
\leq C_2(\Omega_0,\mathbb{R})\|p_{\mathrm{in}/\mathrm{out}}\|_{L^{2}(\Gamma_{\mathrm{in}/\mathrm{out}})}\; \|\uu\|_{H^1(\Omega_0)}\\
&\leq C_3(\Omega_0, \mathbb{R})\left(\frac{\delta}{2}\|\nabla\uu\|^2_{L^2(\Omega_0)}+\frac{1}{2\delta}\|p_{\mathrm{in}/\mathrm{out}}\|^2_{L^{2}(\Gamma_{\mathrm{in}/\mathrm{out}})}\right). \end{aligned}\]
We now use the coercivity \eqref{coerciv1}, \eqref{coerciv2} of the forms \(((\uu,\uu))\) and  \(a(q,q)\)  and choose \(\delta\)  small enough to  get 
\begin{equation*}\label{Zeitintegral2}\begin{aligned}
 &\frac{1}{2}\int_{\Omega_0}(J|\uu|^2)(T)+
\left(c_1- \frac{\delta C_3}{2}\right)\int_0^T \|\nabla \uu\|_{L^2(\Omega_0}^2 dt+c_2\varepsilon \int_0^T\|\nabla q\|_{L^2(\Omega_0)}^2dS(y)dt\\ &\leq \int_0^T\frac{1}{\rho}\left[\int_{\Gamma_{\mathrm{FSI}}^0}-\sigma_s\vec{\nn}\cdot \ww dS(y)+ \frac{C_3}{2\delta}\|p_{\mathrm{in}/\mathrm{out}}\|^2_{L^{2}(\Gamma_{\mathrm{in}/\mathrm{out}})}\right]dt
\end{aligned}\end{equation*} 
Since the right-hand side is bounded by given data, \(J>0\), and \(c_1- \frac{\delta C_3}{2}>0\) for sufficiently small \(\delta\), we arrive at 
\begin{equation}\begin{aligned}\label{Schranke}&\frac{1}{2}\|\uu(T)\|_{L^2(\Omega_0)}^2+C \|\uu\|_{L^2(0,T;H^1(\Omega_0))}^2+ C \|\sqrt{\varepsilon}  q\|^2_{L^2(0,T;H^1(\Omega_0))} \\ & \leq \mathrm{const}\big(\|\sigma_s\|_{L^2(0,T; L^2(\Gamma_{\mathrm{FSI}}^0))},\|p_{\mathrm{in}/\mathrm{out}}\|_{L^2(0,T;L^2(\Gamma_{\mathrm{in/out}}))},\|\ww\|_{L^2(0,T;L^2(\Gamma_{\mathrm{FSI}^0}))} \big).\end{aligned}\end{equation}
Note that \eqref{Schranke} is valid also for any \(0\leq t\leq T\) instead of \(T\), so that the a-priori estimate (\ref{Schranke}),  gives uniform bounds in \(\varepsilon\) for 
\begin{eqnarray} \label{Aestimes_first_u}
\uu_\varepsilon &\in& L^\infty(0,T;L^2(\Omega_0))\cap L^2(0,T;H^1(\Omega_0)), \\ \label{Aestimes_first_q}
\sqrt{\varepsilon}q_\varepsilon &\in& L^2(0,T;H^1(\Omega_0)).
\end{eqnarray}
Finally, from \eqref{Aestimes_first_u} and  \eqref{Aestimes_first_q}  the weak convergence  of  the solution sequences
$\uu_\varepsilon \rightharpoonup  \uu_0$ and $ \ \sqrt{\varepsilon}q_\varepsilon  \rightharpoonup q_0$ 
in \(L^2(0,T;H^1(\Omega_0))\)  follows as  \(\varepsilon\to 0\). The limiting process is shown in  Section \ref{limit}. 

Note that, due to the non-linearity of the problem, the stated weak convergences are not sufficient and strong convergences in appropriate spaces are necessary for the limiting process in the weak formulation \eqref{weakall}. The compactness of the solution sequence for the velocity \(\uu_\varepsilon\) can be proven  based on the equicontinuity estimate, proven in the next section.
The limiting process is then finalised in Section \ref{limit}.

\section{The second a-priori estimates: 'equicontinuity'
}\label{sec:second}
This section is devoted to the equicontinuity estimate for the velocity approximation \(\uu_\varepsilon\) with respect to time, stated in Theorem \ref{thm:Main1}.
We consider the parameter \(\varepsilon\)  to be fixed, but show that the corresponding estimate is independent from \(\varepsilon\).
\begin{theorem}\label{Mainthm}Let the assumptions in Definition \ref{def:weak}  on the given data be satisfied. 
Let \(\uu=\uu_\varepsilon\) be a weak solution to \eqref{weakall},
and in the case of \(n=3\), in addition suppose  that \(\uu_\varepsilon \in L^4(0,T; L^4(\Omega_0))\) is bounded. Then the estimate 
\begin{equation}\label{secondestimate}\int_0^{T-\tau}\int_{\Omega_0}|\sqrt{J(y,s+\tau)}\uu_\varepsilon(y,s+\tau)-\sqrt{J(y,s)}\uu_\varepsilon(y,s)|^2 dyds\leq C\tau\end{equation}
 holds independently on $\varepsilon$ and \(\tau\), whereby the constant $C=C(\mathbb{R}, \mathbb{R}^{-1},\partial_t \mathbb{R},\partial_t \mathbb{R}^{-1},\nabla \mathbb{R},\nabla\mathbb{R}^{-1})$ depends additionally on \(n,\Omega_t\), the given data \(\sigma_s,\ww, p_{\mathrm{in}/\mathrm{out}}\) and in dimension \(n=3\) on the bound of  \(\|\uu_\varepsilon\|_{L^4(0,T,L^4( \Omega_0))}\). 

\end{theorem}

\begin{proof}
We  start averaging the  weak formulation \eqref{weakall},  building the moving average in time,   by considering the time  integration between \(s\), \(s+\tau\), \(s\in (0, T-\tau)\) and test against suitable testfunctions given below, a technique previously used in the works \cite[Appendix A]{anna2},  \cite[Lemma 9]{Grandmont}, \cite[Lemma 10]{Grandmont08}.
The averaged weak formulation for a moving time-interval 
reads
\begin{eqnarray}\label{A.1.5}
&&0= -\int_s^{s+\tau}\int_{\Omega_0} \uu\cdot \partial_t(J\psi)dydt -\int_s^{s+\tau}\int_{\Omega_0}(\nabla \uu)^T\mathbb{R}\ww\cdot\psi dydt\nonumber \\ &&\qquad + \int_{\Omega_0}[J\uu(s+\tau)\psi(s+\tau)-J\uu(s)\psi(s)]dy\\ &&+\int_s^{s+\tau}\Bigg[\hat b(\uu,\uu,\psi)+((\uu,\psi))+\frac{\varepsilon}{\rho^2}a(q,\Phi) +\frac{1}{\rho}\int_{\Omega_0}\div (\mathbb{R}\uu) \Phi dy-\int_{\Omega_0}\frac{q}{\rho}\div(\mathbb{R}\psi)dy\nonumber\\ &&+\frac{1}{\rho}\int_{\Gamma_{\mathrm{FSI}}^0}\sigma_s\vec{\nn}\cdot \psi dS(y)+\int_{\Gamma_{\mathrm{in}/\mathrm{out}}} \left(\frac{p_{\mathrm{in}/\mathrm{out}}}{\rho}\mathbb{R}^T\vec{\nn}\cdot \psi - \frac{(\uu \cdot \psi)}{2}\mathbb{R}\uu\cdot \vec{\nn} \right)dS(y)\Bigg]dt.\nonumber
\end{eqnarray}
Another integration  over time,  \(\int_0^{T-\tau} ds\) applying \eqref{def:viscosform}, \eqref{def:a1} yields
\begin{eqnarray}\label{A.2}
&&-\int_0^{T-\tau}\int_{\Omega_0}[J\uu(s+\tau)\psi(s+\tau;s)-J\uu(s)\psi(s;s)]dyds\nonumber\\&&= \int_0^{T-\tau}\bigg\{\int_s^{s+\tau}\int_{\Omega_0} -\uu\cdot  \partial_t(J\psi)(t;s)dydt -\int_s^{s+\tau}\int_{\Omega_0}(\nabla \uu)^T\mathbb{R}\ww \cdot \psi(t;s)dydt\nonumber
 \\  &&\qquad+\int_s^{s+\tau}\Bigg[\hat b(\uu,\uu,\psi) +\int_{\Omega_0}\frac{1}{\rho}\div (\mathbb{R}\uu) \Phi dy-\int_{\Omega_0}\frac{q}{\rho}\div(\mathbb{R}\psi)dy
 \\ &&+\frac{\mu}{\rho}\int_{\Omega_0}J^{-1}[\mathbb{R}^T\nabla_y\uu+(\mathbb{R}^T\nabla_y\uu)^T]:[\mathbb{R}^T\nabla_y\psi]dy+\frac{\varepsilon}{\rho^2}\int_{\Omega_0}J^{-1}(\mathbb{R}^{T}\nabla q)\cdot (\mathbb{R}^{T}\nabla\Phi)  dy\nonumber\\ &&+\frac{1}{\rho}\int_{\Gamma_{\mathrm{FSI}}^0}\sigma_s\vec{\nn}\cdot \psi dS(y)+\int_{\Gamma_{\mathrm{in}/\mathrm{out}}} \left(\frac{p_{\mathrm{in}/\mathrm{out}}}{\rho}\mathbb{R}^T\vec{\nn}\cdot \psi-\frac{(\uu \cdot \psi)}{2}\mathbb{R}\uu\cdot \vec{\nn} \right)dS(y)\Bigg]dt\bigg\}ds.\nonumber
 \end{eqnarray}
To prove \eqref{secondestimate}, we chose 
special testfunctions:
\begin{equation}\label{TestF}\begin{aligned}\psi(t,y;s)&:=\R^{-1}(y,t)[(\R \uu)(y,s+\tau)-(\R \uu)(y,s)],
\\\Phi(y;s)&=\Phi(t,y;s):=q(y,s)-q(y,s+\tau). 
\end{aligned}\end{equation}
This particular choice of the velocity testfunction \(\psi\) in \eqref{A.2} contains a stationary part of the difference of weak solutions at \(s\) and \(s+\tau\) and a time dependent  factor \(\mathbb{R}^{-1}(y,t)\) (compare  \cite[(A.6)]{anna2}).
 The Piola Transformation of the velocity  preserves the divergence operator \(\div_y\) during the transformation from one set of time dependent coordinates to another \footnote{Note that the divergence operator of the velocity in the domain \(\Omega_t\), \(\div_x v\), is dependent on the domain deformation,  and thus time dependent, see \eqref{eq:divtransform} in the Appendix.} in the sense of \eqref{eq:divtransform}, see also  \cite[Section 6]{Padula}. The pressure testfunction \(\Phi\) is chosen stationary including the time points \(s,s+\tau\).

With this particular choice of testfunctions, in the following,  we will simplify and split the individual terms of \eqref{A.2} into the left-hand side of and  eleven right hand side terms (I) to (XI). We show that the right-hand side terms can be bounded by \(C\tau\) in both cases \(n=2\) and \(n=3\).

The left-hand side of \eqref{A.2} with the testfunctions as in \eqref{TestF} turns into \begin{equation}\begin{aligned}\label{A9links}
&-\int_0^{T-\tau}\int_{\Omega_0}[J\uu(y,s+\tau)\cdot\psi(y,s+\tau;s)-J\uu(y,s)\cdot\psi(y,s;s)]dyds	\\
&= -\int_0^{T-\tau}\int_{\Omega_0}[J\uu(s+\tau)\cdot \uu(s+\tau)-J\uu(s+\tau)\cdot\R^{-1}(s+\tau)(\R \uu)(y,s)\\ & \qquad-J\uu(s)\cdot\R^{-1}(y,s)(\R\uu)(y, s+\tau)+J\uu(s)\cdot \uu(s)] dyds.
 \end{aligned}\end{equation} 
We now shift one of the \(\mathbb{R}^{-1}\) matrices factors to the other side in the scalar products of the second and third term, use the identity \(J\mathbb{R}^{-T}=\mathbb{J}^T\) and the binomial formula to rewrite \eqref{A9links} in order to finally  obtain the left hand side of \eqref{secondestimate}, 
\begin{eqnarray}\label{A9links2}
&&\hspace{-18pt}-\int_0^{T-\tau}\int_{\Omega_0}|\sqrt{J}\uu(s+\tau)|^2+|\sqrt{J}\uu(s)|^2-\mathbb{J}^T\uu(s+\tau)\cdot \R \uu(s)-\mathbb{J}^T\uu(s)\cdot \R\uu(s+\tau)dyds\nonumber\\
&&=-\int_0^{T-\tau}\int_{\Omega_0}|\sqrt{J}\uu(s+\tau)-\sqrt{J}\uu(s)|^2+2\sqrt{J}\uu(s+\tau)\cdot\sqrt{J}\uu(s)\nonumber\\ &&\qquad\qquad-\mathbb{J}^T\uu(s+\tau)\cdot \R \uu(s)-\mathbb{J}^T\uu(s)\cdot \R\uu(s+\tau)dyds\\
&&=-\int_0^{T-\tau}\int_{\Omega_0}|\sqrt{J}\uu(s+\tau)-\sqrt{J}\uu(s)|^2 dyds- (I),\nonumber
\end{eqnarray}
where the remaining term reads
\[\begin{aligned}(I):=
&\int_0^{T-\tau}\int_{\Omega_0}\big[2\sqrt{J}\uu(s+\tau)\cdot\sqrt{J}\uu(s)\\ &-\mathbb{J}^T\uu(s+\tau)\cdot (\sqrt{J}\mathbb{J}^{-1})(\sqrt{J}\uu)(s)-\mathbb{J}^T\uu(s)\cdot(\sqrt{J}\mathbb{J}^{-1})(\sqrt{J}\uu)(s+\tau)\big]dyds\end{aligned}\]
by replacing \(\mathbb{R}=\sqrt{J} \mathbb{J}^{-1}\sqrt{J}\). 

We rewrite the first term in \((I)\) again so that \(\sqrt{J}\mathbb{J}^{-1}\) appears in the second factor 
\[\begin{aligned}&2(\sqrt{J}\uu)(s+\tau)\cdot (\sqrt{J}\uu)(s)\\ &=\uu(s+\tau)\cdot(\mathbb{J}\sqrt{J}\mathbb{J}^{-1})(s+\tau)(\sqrt{J} \uu)(s)+ \uu(s)\cdot (\mathbb{J}\sqrt{J}\mathbb{J}^{-1})(s)(\sqrt{J}\uu)(s+\tau)\\&=(\mathbb{J}^T\uu)(s+\tau)\cdot \sqrt{J}\mathbb{J}^{-1}(s+\tau)(\sqrt{J}\uu)(s)+(\mathbb{J}^T\uu)(s)\cdot \sqrt{J}\mathbb{J}^{-1}(s)(\sqrt{J}\uu)(s+\tau)\end{aligned}\]
and introduce the matrix \(\mathbb{M}(y,s):=\int_s^{s+\tau} (\sqrt{J} \mathbb{J}^{-1})(y,z)dz\) with the property   \[\frac{\partial}{\partial s}\mathbb{M}(y,s)=(\sqrt{J} \mathbb{J}^{-1})(y,s+\tau)-(\sqrt{J} \mathbb{J}^{-1})(y,s)=\int_s^{s+\tau}\frac{\partial}{\partial z}\left((\sqrt{J} \mathbb{J}^{-1})(y,z)\right) dz\] using the Leibnitz-formula for derivatives of integrals and the fundamental theorem. So \((I)\) can be rearranged as
\[\begin{aligned}(I)
&=\int_0^{T-\tau}\int_{\Omega_0}(\mathbb{J}^T\uu)(s+\tau)\cdot(\partial_s\mathbb{M}\sqrt{J}\uu)(s)-(\partial_s\mathbb{M}^T\mathbb{J}^T\uu)(s)\cdot(\sqrt{J}\uu)(s+\tau) dyds.\end{aligned}\]
The estimate for the Frobenius-norm of \(\partial_s\mathbb{M}\) yields
 \[\begin{aligned}|\partial_s\mathbb{M}(s)|_F\leq  \int_s^{s+\tau}\left|\frac{\partial}{\partial_z}\left(\sqrt{J(z)}\mathbb{J}^{-1}(z)\right)\right|_Fds\leq C(\mathbb{R},\mathbb{R}^{-1},\partial_t\mathbb{R}, \partial_t\mathbb{R}^{-1})\int_s^{s+\tau} 1 ds \leq C\tau,\end{aligned}\] 
since 
\(\mathbb{J}^{-1}\) can be explicitly written in terms of \(\dd\in C^2(0,T,C^2(\overline{\Omega_0}))\) and  \(\nabla_y \dd\) both having bounded time-derivatives. 
Correspondingly for fixed finite \(T\) and bounded \(\Omega_0\) one has that 
\begin{equation}\label{est:III}\begin{aligned}|(I)|&\leq \int_0^{T-\tau}\int_{\Omega_0}\big(|\mathbb{J}^T|_F|\uu|\big)(s+\tau)\;\big(|\partial_s\mathbb{M}|_F|\sqrt{J}|\; |\uu|\big)(s)\\&\qquad+\big(|\partial_s\mathbb{M}^T|_F|\mathbb{J}^T|_F|\uu|\big)(s)\;\big(|\sqrt{J}|\;|\uu|\big)(s+\tau) dyds\\
&\leq C(\mathbb{R},\mathbb{R}^{-1},\partial_t \mathbb{R}, \partial_t \mathbb{R}^{-1})\tau\|\uu\|_{L^\infty(0,T;L^2(\Omega_0))}^2.\end{aligned}\end{equation}
The next term  related to the time derivative of \(J\psi\) in \eqref{A.2}, with the particular choice of testfunctions  \eqref{TestF}, is given by
\[\begin{aligned}(II)&:=\int_0^{T-\tau}\int_s^{s+\tau}\int_{\Omega_0}  \uu\cdot \partial_t(J\psi)(y,t;s)dydtds\\ & =  \int_0^{T-\tau}\int_s^{s+\tau}\int_{\Omega_0} \uu(t)\cdot \partial_t \mathbb{J}(t)[(\mathbb{R}\uu)(s+\tau)-(\mathbb{R}\uu)(s)]dydtds, \end{aligned}\] 
where we recall \(J\mathbb{R}^{-1}=\mathbb{J}\) and \(|\partial_t\mathbb{J}|_F\leq C(\mathbb{R},\mathbb{R}^{-1},\partial_t\mathbb{R}, \partial_t\mathbb{R}^{-1})=:C(\mathbb{R},\partial_t\mathbb{R})\) holds in short notation by the assumptions, see \eqref{ass:d} and Lemma \ref{invertible}. We estimate \((II)\) using the Cauchy-Schwarz-inequality 
\begin{equation}\label{est:IV}\begin{aligned}|(II)|
& \leq C(\mathbb{R},\partial_t\mathbb{R}) \int_0^{T-\tau}\left[\|\uu(s+\tau)\|_{L^2(\Omega_0)}+\|\uu(s)\|_{L^2(\Omega_0)}\right]\int_s^{s+\tau}  \|\uu(t)\|_{L^2(\Omega_0)}\;dtds\\
& \leq C(\mathbb{R},\partial_t\mathbb{R}) \tau \|\uu\|_{L^\infty(0,T;L^2(\Omega_0))}^2 .\end{aligned}\end{equation}
Next, we turn to the divergence terms on the right hand side of \eqref{A.2}. We factor out \(\frac{1}{\rho}\) and interchange the integration w.r.t. \(t,s\) 
\[\begin{aligned}\int_0^{T-\tau} \int_{\Omega_0}\int_s^{s+\tau}[\div_y (\mathbb{R}\uu) \Phi(s)- q\div_y(\mathbb{R}\psi(t;s))]dtdyds.
\end{aligned}\] 
Note that by the choice of the testfunction  \eqref{TestF},  \(\Phi\) does not depend on the time \(t\) and thus can be pulled out of the inner time integral. Moreover, the \(\div_y\)-operator is time- independent as well, and can be put also in front of that integral.
Inserting correspondingly \(\psi\),  the factor \(\mathbb{R}(y,t)\) cancels in the second term, so that  another time-independent term can be moved outside the  integral w.r.t. the time \(t\) and we get 
\[\begin{aligned}&\int_0^{T-\tau} \int_s^{s+\tau}\int_{\Omega_0}[\div_y (\mathbb{R}\uu) \Phi(s)- q\div_y(\mathbb{R}\psi(t;s))]dydtds\\
&=\int_0^{T-\tau}\int_{\Omega_0}\left[\div_y\left(\int_s^{s+\tau} (\mathbb{R}\uu)(y,t)dt\right)(q(y,s)-q(y,s+\tau))\right.\\ &\qquad\left. -\big[\div_y\big((\R \uu)(y,s+\tau)-(\R \uu)(y,s)\big)\big]\int_s^{s+\tau}q(t)dt\right]dyds.
\end{aligned}\]
By the following  abbreviations for the time integrals \[U(y,s):=\int_s^{s+\tau}\R(y,t)\uu(y,t)dt,\qquad Q(y,s):= \int_s^{s+\tau}q(y,t)dt,\]
with the property \(\frac{\partial}{\partial_s}U(y,s)=(\R \uu)(y,s+\tau)-(\R \uu)(y,s), \frac{\partial}{\partial_s}Q(y,s)=q(y,s+\tau)-q(y,s)\), that is valid almost everywhere by the Lebesgue Differentiation Theorem, we can reformulate the  difference of the divergence term as follows
\[\begin{aligned}&\int_0^{T-\tau} \int_s^{s+\tau}\int_{\Omega_0}\left[\div_y (\mathbb{R}\uu) \Phi(s)- q\div_y(\mathbb{R}\psi(t;s))\right]dydtds\\
&=-\int_0^{T-\tau}\int_{\Omega_0}\left[\frac{\partial}{\partial_s}Q(s)\div_yU(s)+ Q(s)\div_y\frac{\partial}{\partial_s}U(s)\right]dyds\\&= \int_0^{T-\tau}\int_{\Omega_0}-\frac{\partial}{\partial_s}(Q \div_y U)(s) dyds\qquad 
\\&=
\int_{\Omega_0}\big([Q \div_y U](0)-[Q \div_y U](T-\tau)\big) dy
\\&
=\int_{\Omega_0} \int_0^\tau q(t)dt\;\int_0^\tau \div_y (\R \uu)(t)dtdy-\int_{T-\tau}^Tq(t) dt\; \int_{T-\tau}^T\div_y(\R \uu) (t)dt dy\\
&=: (III)-(IV), \end{aligned}\]
where we replaced \(Q(0)=\int_0^\tau q(t)dt\), \(Q(T-\tau)=\int_{T-\tau}^Tq(t)dt\) and analogusly for \(U\).
The terms \((III), (IV)\) differ only in the time interval of size \(\tau\) considered, so that they can be treated in the same way. The first integral, which is independent on \(t\),
\[(III)=\int_{\Omega_0}Q(0)\int_0^\tau \div_y (\R \uu)(t)dt dy\] 
shows up in the weak formulation \eqref{A.1.5} for \(s=0\) if we  multiply with \(\rho\) and test with  a zero-velocity testfunction \(\psi\) and  constant averaged pressure \(\Phi\): 
\begin{equation}\label{testp1}\;\psi:=0,\; \Phi(y):=Q(y,0)=\int_0^\tau q(y,t) dt.\end{equation} 
In this case we get from \eqref{A.1.5} using \eqref{def:a1} that 
\[\begin{aligned}0&= \frac{\varepsilon}{\rho} \int_0^\tau a\left(q(y,t), Q(y,0) \right) dt+\int_0^\tau\int_{\Omega_0}\div_y(\R \uu)Q(y,0)dy dt
\\ &=:\int_0^\tau \frac{\varepsilon}{\rho} \int_{\Omega_0}J^{-1}\mathbb{R}^{T}\nabla q(y,t)\cdot(\mathbb{R}^{T}(y,t) \nabla Q(y,0) )dydt+(III)
.\end{aligned}\]
Consequently, we can estimate \((III)\) by  sorting and using the Cauchy-Schwarz inequality,  \[\begin{aligned}|(III)|&\leq\frac{1}{\rho}\int_0^\tau \|\R\|_\infty^2 \|J^{-1}\|_\infty \|\sqrt{\varepsilon} \nabla_yq\|_{L^2(\Omega_0)}\; \left\|\int_0^\tau\sqrt{\varepsilon}  \nabla_y q dt \right\|_{L^2(\Omega_0)} dt\\
&\leq  C(\mathbb{R}) \tau^{1/2}\|\sqrt{\varepsilon} \nabla_y q\|_{L^2(0,T; L^2(\Omega_0))}\; \left\| \int_0^\tau \sqrt{\varepsilon}\nabla_y q dt \right\|_{L^2(\Omega_0)}.
\end{aligned}\]
For the square of the last factor the Cauchy-Schwarz inequality implies 
\[\begin{aligned}\left(\left\|\int_0^\tau \sqrt{\varepsilon}\nabla_y q dt\right\|_{L^2(\Omega_0)}\right)^2
\leq \tau\|\sqrt{\varepsilon}q\|_{L^2(0,T; H^1(\Omega_0))}^2\end{aligned}\]
and get finally
\begin{equation}\label{est:I+II}|(III)|\leq\tau C(\mathbb{R})\|\sqrt{\varepsilon}q\|_{L^2(0,T; H^1(\Omega_0))}^2.\end{equation}
Recall that the boundedness of \(\|\sqrt{\varepsilon}q\|_{L^2(0,T; H^1(\Omega_0))}\|\) follows from the first a-priori estimate \eqref{Schranke}. The estimate \(|(IV)|\leq\tau  C(\mathbb{R})\|\sqrt{\varepsilon}q\|_{L^2(0,T; H^1(\Omega_0))}^2\) can then be obtained in the same way, with \((T-\tau,T)\) as interval of size \(\tau\) instead of \((0,\tau)\). \medskip

We now turn to the estimate of the viscous and convective terms in \eqref{A.2} with testfunctions \eqref{TestF}. 
\[\begin{aligned}(V)&:=((\uu, \psi))\\ &=
\frac{\mu}{\rho} \int_0^{T-\tau}\int_s^{s+\tau} \int_{\Omega_0}J^{-1}[\mathbb{R}^T\nabla_y \uu+(\mathbb{R}^T\nabla_y \uu)^T]:[\mathbb{R}^T\nabla_y \psi](y,t;s)dydtds.\end{aligned}\]
  To express the gradient the special testfunction \(\psi \) in \eqref{TestF},   we use the product-rule twice for \(\theta=s\) or \(\theta=s+\tau\), respectively to calculate the components of the matrix \((\nabla_y \psi)_{i,j}\):
\[\begin{aligned}&\partial_{y_i}\big[\R^{-1}(y,t)(\R \uu)(y,\theta)\big]_j
=\sum_{k,l}\Big( (\partial_{y_i} \mathbb{R}^{-1}_{jk})(y,t)[\mathbb{R}_{kl}(y,\theta) \uu_l(y,\theta)]\\ & \qquad +  (\mathbb{R}^{-1}_{jk})(y,t)(\partial_{y_i} \mathbb{R}_{kl})(y,\theta)\uu_l(y,\theta) +   (\mathbb{R}^{-1}_{jk})(y,t)\mathbb{R}_{kl}(y,\theta) (\partial_{y_i}\uu_l)(y,\theta) \Big).
\end{aligned}\]
By assumption on the deformation \(\dd\), all components of \(\mathbb{R},\mathbb{R}^{-1}\) and the corresponding derivatives with respect to \(y_i\) are all bounded. One then has
\[|\partial_{y_i}[\R^{-1}(y,t)(\R \uu)(y,\theta)]_j|\leq C \sum_l |\uu_l(y,\theta)|+|\partial_{y_i}\uu_l(y,\theta)|\] with \(C=C(\mathbb{R}, \mathbb{R}^{-1}, \nabla \mathbb{R}, \nabla \mathbb{R}^{-1})=:C(\mathbb{R}, \nabla \mathbb{R})\) in short notation and consequently \begin{equation}|\label{est:psi}\nabla_y\psi(y,t;s)|_F\leq C(|\uu(y,s)|+|\uu(y,s+\tau)|+|\nabla_y \uu(y,s)|_F+|\nabla_y\uu(y,s+\tau)|_F
).\end{equation}
We use the boundedness of \(J^{-1}\), the individual entries of \(\mathbb{R}\) and the above estimates to conclude:
\[\begin{aligned}&|(V)| \leq C(\mathbb{R})\int_0^{T-\tau}\int_s^{s+\tau}\int_{\Omega_0}|\mathrm{trace}([\mathbb{R}^T\nabla_y \uu+(\mathbb{R}^T\nabla_y \uu)^T]\mathbb{R}^T\nabla_y \psi)|dydtds\\
&\leq C(\mathbb{R})\int_0^{T-\tau}\int_s^{s+\tau}\int_{\Omega_0}|\nabla_y \uu|_F|\nabla_y \psi|_Fdydtds\\
&\leq C(\mathbb{R}, \nabla \mathbb{R})\int_0^{T-\tau}\int_s^{s+\tau} \|\nabla_y\uu(t)\|_{L^2(\Omega_0)}\Big(\|\uu(s)\|_{H^1(\Omega_0)}+\|\uu(s+\tau)\|_{H^1(\Omega_0)}\Big)dtds\\
&=C(\mathbb{R}, \nabla \mathbb{R})\int_0^{T-\tau}  \Big(\|\uu(s)\|_{H^1(\Omega_0)}+\|\uu(s+\tau)\|_{H^1(\Omega_0)}\Big) \tau\frac{1}{\tau} \int_s^{s+\tau}\|\nabla_y\uu(t)\|_{L^2(\Omega_0)} dt ds.
\end{aligned}\]
We introduce \([\varphi]_\tau(s):=\frac{1}{\tau}\int_s^{s+\tau} \varphi(t)dt\) as the Steklov-average of \(\varphi\). 
We apply H\"older twice and use the Steklov-average property \eqref{stek1} for \(\varphi=\| \nabla_y \uu\|_{L^2(\Omega_0)}\)  to further estimate \((V)\) by
\begin{equation}\label{est:V}
\begin{aligned}|(V)| 
\leq C(\mathbb{R}, \nabla \mathbb{R})\tau \|\uu\|_{L^2(0,T;H^1(\Omega_0))}^2 .\end{aligned}\end{equation}
\medskip
Next, we estimate the nonlinear \(\hat{b}\)-term together with the boundary term over  \(\Gamma_{\mathrm{in/out}}\), which arises from the dynamic pressure condition. Note that this boundary integral in the weak formulation \eqref{A.2} is cancelled due to the partial integration, see \eqref{bterm}. The resulting nonlinear terms are split into volume and boundary terms, i.e.,  \[\int_0^{T-\tau}\int_s^{s+\tau} \hat b(\uu,\uu,\psi)-\frac{1}{2}\int_{\Gamma_{\mathrm{in/out}}} (\uu \cdot \psi)\mathbb{R}\uu\cdot \vec{\nn} dS(y)dtds=:(VI)+(VII),\] 
where
\[(VI):=\int_0^{T-\tau}\int_s^{s+\tau}\int_{\Omega_0}\frac{1}{2}[(\nabla_y\uu)^T \mathbb{R}\uu\cdot \psi-(\nabla_y \psi)\mathbb{R}\uu\cdot \uu]dydtds\]
and 
\[(VII):=\int_0^{T-\tau}\int_s^{s+\tau}\frac{1}{2}\int_{\Gamma_{\mathrm{FSI}}^0}(\ww\cdot \psi)\mathbb{R}\ww\cdot \vec{\nn} dS(y)dtds\] by
using \(\uu=\ww\) on \(\Gamma_{\mathrm{FSI}}^0\), \(\uu=0\) on \(\Gamma_{\mathrm{wall}}\), cf. the boundary terms in \eqref{bterm}.
We use the boundedness of \(\mathbb{R},  \mathbb{R}^{-1}\) and their space-derivatives, cf. \eqref{est:psi}, and the H\"older inequality to estimate correspondingly 
\begin{equation}\label{**}\begin{aligned}&|(VI)|\leq C(\mathbb{R}) \int_0^{T-\tau}\int_s^{s+\tau}\int_{\Omega_0} |\uu|\;\Big((\nabla_y \uu)|_F |\psi| +|(\nabla_y \psi)|_F |\uu| \Big)dydtds
\\ & \leq C(\mathbb{R}, \nabla \mathbb{R}) \int_0^{T-\tau}\int_s^{s+\tau} \|\uu(t)\|_{L^4(\Omega_0)}\;\|\nabla_y\uu(t)\|_{L^2}\; \Big(\|\uu(s+\tau)\|_{L^4}+\|\uu(s)\|_{L^4}\Big)\\ & \qquad +\|\uu(t)\|_{L^4(\Omega_0)}\;\|\nabla_y \psi(t)\|_{L^2}\;\|\uu(t)\|_{L^4}dtds.
\end{aligned}\end{equation}
Due to the different Sobolev-imbeddings  \eqref{n=2,Omega0}, \eqref{n=3,Omega0} we  have to distinguish the cases of dimension \(n=2\) and \(3\).
Recall that for the case \(n=3\), we additionally assume
\begin{equation}\label{uniformL4L4}\|\uu\|_{L^4(0,T;L^4(\Omega_0))}\leq C,
\end{equation}
which holds automatically in 
the case of \(n=2\) by the a-priori estimate in \eqref{Aestimes_first_u} and the Sobolev-imbedding \eqref{n=2,Omega0}.

We start with the case of \(n=2\). Set \(\theta\in \{s,s+\tau\}\) arbitrary  to estimate the first summand of \eqref{**} by the imbedding of \(L^4(\Omega_0)\), cf. \eqref{n=2,Omega0}, as follows:
\[\begin{aligned}&\int_0^{T-\tau}\int_s^{s+\tau} \|\uu(t)\|_{L^4(\Omega_0)}\;\|\nabla\uu(t)\|_{L^2}\;\|\uu(\theta)\|_{L^4}dtds
\\ & \leq C(\Omega_0)\|\uu\|_{L^\infty(0,T; L^2(\Omega_0))}\; \int_0^{T-\tau}\int_s^{s+\tau} \|\uu(t)\|_{H^1(\Omega_0)}^{1/2} \; \|\uu(t)\|_{H^1} \; \|\uu(\theta)\|_{H^1}^{1/2}dtds.
\end{aligned}\]
Now, we apply the Young-inequality \(|ab|\leq \frac{1}{p}|a|^p+\frac{p-1}{p}|b|^{\frac{p}{p-1}},  1<p<\infty\) with \\ \(a=\|\uu (t)\|_{H^1(\Omega_0)}^{1/2}, b=  \; \|\uu(t)\|_{H^1(\Omega_0)}^{3/2}\) and \(p=4\), to obtain 
\[\begin{aligned}&\int_0^{T-\tau}\int_s^{s+\tau} \|\uu(t)\|_{H^1(\Omega_0)}^{3/2} \; \|\uu(\theta)\|_{H^1}^{1/2}dtds\leq 
\int_0^{T-\tau}\int_s^{s+\tau}\frac{1}{4}\| \uu(\theta)\|_{H^1(\Omega_0)}^2+\frac{3}{4}\| \uu(t)\|_{H^1}^{2}dtds \\ &  \leq \frac{1}{4}\int_0^{T-\tau}\| \uu(\theta)\|_{H^1}^{2} \int_s^{s+\tau} 1dt ds+\frac{3}{4}\tau\int_0^{T-\tau}\frac{1}{\tau}\int_s^{s+\tau}\|\uu(t)\|_{H^1(\Omega_0)}^2dtds.
\end{aligned}\]
The first term is bounded by \(\frac{\tau}{4}\|\uu\|_{L^2(0,T,H^1(\Omega_0))}\), the second is estimated  with the help of  the Steklov-average property \eqref{stek1} for \(\varphi(t):=\|\uu(t)\|_{H^1(\Omega_0)}^2\). We get in the second term 
\begin{equation}\begin{aligned}
\label{est:stek}\int_0^{T-\tau}\frac{1}{\tau}\int_s^{s+\tau}\|\uu(t)\|_{H^1(\Omega_0)}^2dtds&=\|[\varphi]_\tau\|_{L^1(0,T-\tau)}\\ &\leq \|\varphi\|_{L^1(0,T)}= \|\uu\|_{L^2(0,T; H^1(\Omega_0))}^2.\end{aligned}\end{equation} By  \eqref{est:stek},  the first summand of \eqref{**} can  be estimated  for \(n=2\) by 
\[\begin{aligned}&C(\mathbb{R})\int_0^{T-\tau}\int_s^{s+\tau} \|\uu(t)\|_{L^4(\Omega_0)}\; \|\nabla\uu(t)\|_{L^2}\; \|\uu(\theta)\|_{L^4}dtds\\ &\leq C(\Omega_0,\mathbb{R})\tau \|\uu\|_{L^\infty(0,T; L^2(\Omega_0))} \; \|\uu\|_{L^2(0,T;H^1(\Omega_0))}^2.\end{aligned}\]
The second summand of \eqref{**} can be estimate in a similar way using \eqref{n=2,Omega0}, \eqref{est:psi}
\[\begin{aligned}& C(\mathbb{R})\int_0^{T-\tau}\int_s^{s+\tau}\|\uu(t)\|_{L^4(\Omega_0)}^2\; \|\nabla\psi(t)\|_{L^2}dtds\\ &
\leq C(\Omega_0,\mathbb{R})\int_0^{T-\tau}\int_s^{s+\tau}\|\uu(t)\|_{L^2(\Omega_0)}\; \|\nabla \uu(t)\|_{L^2}\; \|\nabla\psi(t)\|_{L^2} dtds
\\& \leq C(\Omega_0, \mathbb{R}, \nabla \mathbb{R})\|\uu\|_{L^\infty(0,T; L^2(\Omega_0))} \\ &\hspace{90pt} \int_0^{T-\tau}\int_s^{s+\tau}\|\nabla  \uu (t)\|_{L^2(\Omega_0)} \;(\| \uu(s+\tau)\|_{H^1}+\| \uu(s)\|_{H^1}) dtds.
\end{aligned}\]
Here we use the Young-inequality for \(p=2\) and the H\"older-inequality to get correspondingly using a similar Steklov-average property as in \eqref{est:stek}
\[\begin{aligned} &\int_0^{T-\tau}\int_s^{s+\tau}\|\nabla \uu (t)\|_{L^2(\Omega_0)} \;\big(\| \uu(s+\tau)\|_{H^1(\Omega_0)}+\| \uu(s)\|_{H^1(\Omega_0)}\big) dtds\\ &
\leq \int_0^{T-\tau}\int_s^{s+\tau}\| \uu(t)\|_{H^1(\Omega_0)}^2+\frac{1}{2}\|\uu(s+\tau)\|_{H^1(\Omega_0)}^2+\frac{1}{2}\| \uu(s)\|_{H^1(\Omega_0)}^2
dtds\\ & \leq 2\tau  \; \|\uu\|_{L^2(0,T;H^1(\Omega_0))}^2.\end{aligned}\]
The whole estimate for \((VI)\) in the case \(n=2\) reads in the short notation 
\begin{equation}\label{est:6n=2}
|(VI)|\leq C(\Omega_0,\mathbb{R}, \nabla \mathbb{R}) \tau \big(\|\uu\|_{L^\infty(0,T; L^2(\Omega_0))} \|\uu\|_{L^2(0,T; H^1(\Omega_0))}^2\big).
\end{equation}
In the case \(n=3\) we proceed in a similar way using the  \(L^4-\)boundedness of \(\uu\), the Steklov properties as in \eqref{est:stek}, applying \eqref{stek1} 
with \(\varphi(t)=\| \uu(t)\|_{L^4(\Omega_0)}^4 \) and the Young inequality with \(p=2\) twice to estimate the first summand in \eqref{**}, 
\begin{equation}
\begin{aligned}&C(\mathbb{R})\int_0^{T-\tau}\int_s^{s+\tau} \|\uu(t)\|_{L^4(\Omega_0)}\; \|\nabla\uu(t)\|_{L^2}\; \|\uu(\theta)\|_{L^4}dtds\\ &\leq C(\mathbb{R}) \int_0^{T-\tau}\int_s^{s+\tau} \frac{1}{2}\|\uu(t)\|_{H^1(\Omega_0)}^2+\frac{1}{4}\|\uu(t)\|_{L^4}^4+ \frac{1}{4}\|\uu(\theta)\|_{L^4}^4 dtds\\ &
\leq C(\mathbb{R})\tau \big(\|\uu\|_{L^2(0,T;H^1(\Omega_0))}^2+\|\uu\|_{L^4(0,T;L^4(\Omega_0))}^4\big).
\end{aligned}\end{equation}
The second summand of \eqref{**} is estimated similarly  with \eqref{est:psi}  and applying  \eqref{stek1} for \(\varphi=\|\uu\|_{L^4(\Omega_0)}^4\), \[\begin{aligned}& C(\mathbb{R})\int_0^{T-\tau}\int_s^{s+\tau}\|\uu(t)\|_{L^4(\Omega_0)}^2\; \|\nabla\psi(t)\|_{L^2}dtds\\ &
\leq C(\mathbb{R})\int_0^{T-\tau}\int_s^{s+\tau}\frac{1}{2}\|\uu(t)\|_{L^4(\Omega_0)}^4+\frac{1}{2}\; \|\nabla\psi(t)\|_{L^2}^2dtds
\\&\leq C(\mathbb{R},\nabla\mathbb{R})\int_0^{T-\tau}\int_s^{s+\tau}\|\uu(t)\|_{L^4(\Omega_0)}^4+\|\uu(s)\|_{H^1}^2+\|\uu(s+\tau)\|_{H^1}^2\; dtds\\
& \leq C(\mathbb{R},\nabla \mathbb{R})\tau \big(\|\uu\|_{L^2(0,T;H^1(\Omega_0))}^2+\|\uu\|_{L^4(0,T;L^4(\Omega_0))}^4\big).
\end{aligned}\]
The whole estimate of \((VI)\) in the case \(n=3\) reads finally 
\begin{equation}
|(VI)|\leq C(\mathbb{R},\mathbb{R}^{-1},\nabla \mathbb{R},\nabla \mathbb{R}^{-1})\tau \big(\|\uu\|_{L^2(0,T;H^1(\Omega_0))}^2+\|\uu\|_{L^4(0,T;L^4(\Omega_0))}^4\big)\end{equation}
Note that the power 4 of \(\|\uu\|_{L^4(\Omega_0)}\) in the estimates above  arises naturally, however, by the Sobolev-imbedding for \(n=3\) in \eqref{n=3,Omega0}, one would obtain  \(\|\uu\|_{L^4(\Omega_0)}^4\leq C\|\uu\|_{L^2(\Omega_0)} \|\uu\|_{H^1(\Omega_0)}^3\), which is not bounded in the weak solution space  \(L^\infty(0,T;L^2(\Omega_0))\cap L^2(0,T;H^1(\Omega_0))\),  see \eqref{assu}. Thus \(\uu\in L^4(0,T,L^4(\Omega_0))\) is a reasonable necessary assumption in Theorem \ref{Mainthm}. 

We now focus on the boundary part \((VII)\) of the nonlinear  term  \(\hat b(\uu,\uu,\psi)\).
Using our assumptions on the boundedness of \(\mathbb{R},\mathbb{R}^{-1}\) in \(\psi\) defined by \eqref{TestF} and the H\"oder inequality, one has 
\[\begin{aligned} |&(VII)|\leq C(\mathbb{R})  \int_0^{T-\tau}\int_s^{s+\tau} \|\ww(t)\|_{L^3(\Gamma_{\mathrm{FSI}}^0)}\;\|\psi(t)\|_{L^3(\Gamma_{\mathrm{FSI}}^0)}\;\|\ww(t)\|_{L^3(\Gamma_{\mathrm{FSI}}^0)}dtds\\ &
\leq C(\mathbb{R},\nabla \mathbb{R})  \int_0^{T-\tau}\int_s^{s+\tau} \|\ww(t)\|_{L^3(\partial \Omega_0)}^2\;\Big(\|\ww(s+\tau)\|_{L^3(\partial  \Omega_0)}+\|\ww(s)\|_{L^3(\partial  \Omega_0)}\Big)dtds.
\end{aligned}\]
Here, we can treat the cases \(n=2\) and \(n=3\) simultaneously since by imbedding \(\|\ww\|_{L^3(\partial \Omega_0)}\leq C(\Omega_0) \|\ww\|_{H^1(\Omega_0)}\), see \eqref{defW},  \eqref{n=2,deltaOmega3}, \eqref{n=3,deltaOmega4}. 
We use the Young-inequality with \(p=3\)  twice and the Steklov-property \eqref{stek1}  for \(\varphi(t)=\|\ww(t)\|_{H^1(\Omega_0)}^3\) to get

\begin{eqnarray}\label{est:VII} &\hspace{-9pt}|(VII)|& \leq C(\Omega_0, \mathbb{R},\nabla \mathbb{R}) \int_0^{T-\tau}\int_s^{s+\tau} \frac{4}{3}\|\ww(t)\|_{H^1(\Omega_0)}^{3} + \frac{\|\ww(s+\tau)\|_{H^1}^3}{3}+\frac{\|\ww(s)\|_{H^1}^3}{3}dtds\nonumber\\
&&  \leq C(\Omega_0,\mathbb{R},\nabla \mathbb{R})\tau \|\ww\|_{L^3(H^1(\Omega_0))}^{3}\leq C(\Omega_0, \mathbb{R})\tau \|\ww\|_{C^1(0,T;C^2(\overline{\Omega_0}))}.
\end{eqnarray}  
This concludes the estimate of the convective term and its boundary part.
Next, we set for the ALE-term coming from the transformation of the time derivative
\[\begin{aligned}(VIII) &:= \int_0^{T-\tau}\int_s^{s+\tau}\int_{\Omega_0} (\nabla_y \uu)^T \mathbb{R}\ww\cdot \psi dydtds.\end{aligned}\]
By the H\"older inequality and the boundedness of \(\mathbb{R}\) and  \(\|\ww\|_{C(0,T,L^\infty(\overline{\Omega_0}))}\),  we have the estimate
\[\begin{aligned} &|(VIII)|\leq \int_0^{T-\tau}\int_s^{s+\tau}\int_{\Omega_0} |\nabla_y\uu|_F\; |\mathbb{R}|_F \; |\ww|\;|\psi| dydtds
\\&\leq C(\mathbb{R},\mathbb{R}^{-1},\ww)\int_0^{T-\tau}\int_s^{s+\tau}\|\uu(t)\|_{H^1(\Omega_0)}\Big(\|\uu(s+\tau)\|_{L^2(\Omega_0)}+\|\uu(s)\|_{L^2(\Omega_0)}\Big)\; dtds\\
&\leq C(\mathbb{R},\mathbb{R}^{-1},\ww)\tau\int_0^{T-\tau}\Big(\|\uu(s+\tau)\|_{L^2(\Omega_0)}+\|\uu(s)\|_{L^2(\Omega_0)}\Big) \left[\|\uu\|_{H^1(\Omega_0)}\right]_\tau(s) ds.\\
\end{aligned}\]
Here we apply the Steklov-average property \eqref{stek1}  for \(\varphi(t):=\|\uu(t)\|_{H^1(\Omega_0)}\) and the H\"older-inequality to arrive at \begin{eqnarray}\label{est:VIII} &\hspace{-10pt}|(VIII)|&\leq C(\mathbb{R},\ww) \tau\big(\|\uu\|_{L^2(\tau,T;L^2(\Omega_0))}+\|\uu\|_{L^2(0,T-\tau;L^2(\Omega_0))}\big)\;\|\left[\|\uu\|_{H^1(\Omega_0)}\right]_\tau\|_{L^2(0,T-\tau)}\nonumber
\\ &&\leq C(\mathbb{R},\mathbb{R}^{-1},\ww) \tau\|\uu\|_{L^2(0,T;H^1(\Omega_0))}^2.
\end{eqnarray}

The next term associated with the pressure Laplacian, cf. \eqref{def:a1}, is denoted (up to \(\rho^2\)) by \[\begin{aligned}(IX):= \varepsilon  a(q,\Phi)=\varepsilon \int_0^{T-\tau}\int_s^{s+\tau}\int_{\Omega_0}J^{-1}(\mathbb{R}^{T}\nabla_y q)\cdot (\mathbb{R}^{T}\nabla_y \Phi)dydtds,\end{aligned}\]  where  \(\Phi(y;s):=q(y,s)-q(y,s+\tau) \). Analogusly we set \(\varphi(t):=\|q(t)\|_{H^1(\Omega_0)}\) and use the Steklov-average property \eqref{stek1} with \(p=2\) to estimate
\begin{eqnarray}   
\label{est:IX}& |(IX)|& \leq C(\mathbb{R})\varepsilon \int_0^{T-\tau}\Big(\|q(s)\|_{H^1(\Omega_0)}+\|q(s+\tau)\|_{H^1(\Omega_0)}\Big)\int_s^{s+\tau}\|q(t)\|_{H^1(\Omega_0)} dtds \nonumber\\
& &\leq C(\mathbb{R})\varepsilon \tau \int_0^{T-\tau}\Big(\|q(s)\|_{H^1(\Omega_0)}+\|q(s+\tau)\|_{H^1(\Omega_0)}\Big)\;  \left[\|q\|_{H^1(\Omega_0)}\right]_\tau(s)ds\\
&&\leq  C(\mathbb{R})\varepsilon \tau \Big(\|q\|_{L^2(0, T-\tau;H^1(\Omega_0))}+\|q\|_{L^2(\tau,T;H^1(\Omega_0))}\Big)\|\left[\|q\|_{H^1(\Omega_0)}\right]_\tau\|_{L^2(0,T-\tau)}\nonumber\\ 
&&\leq C(\mathbb{R}) \tau \|\sqrt{\varepsilon}q\|_{L^2(0,T;H^1(\Omega_0))}^2.\nonumber
\end{eqnarray}

The next remaining term is related to the tension (external stress) acting at the fluid-solid interface
\[\begin{aligned}(X)
&:=\frac{1}{\rho}\int_0^{T-\tau}\int_s^{s+\tau} \int_{\Gamma_{\mathrm{FSI}}^0}
 \sigma_s\vec{\nn}\cdot \psi dS(y)dtds. 
\end{aligned}\] 
 With the boundedness of \(\mathbb{R}, \mathbb{R}^{-1}\) and the coupling condition \eqref{eq:w}, one has
\[\begin{aligned}&|(X)|
 \leq \frac{1}{\rho}\int_0^{T-\tau}\int_s^{s+\tau}\int_{\Gamma_{\mathrm{FSI}}^0}|\psi|\; |\sigma_s|_FdS(y)dtds
\\
&\leq C(\mathbb{R},\nabla \mathbb{R})\tau\int_0^{T-\tau}\big(\|\ww(s+\tau)\|_{L^2(\Gamma_{\mathrm{FSI}}^0)}+\|\ww(s)\|_{L^2(\Gamma_{\mathrm{FSI}}^0)}\big)\left[\|\sigma_s(t;s)\|_{L^2(\Gamma_{\mathrm{FSI}}^0)}\right]_\tau ds.
\end{aligned}\]
Using the H\"older inequality and the property of the Steklov-average  \eqref{stek1} for the function \(\varphi(t):=\|\sigma_s(t)\|_{L^2(\Gamma_{\mathrm{FSI}}^0)}\) and the boundary imbeddings \eqref{n=2,deltaOmega2}, \eqref{n=3,deltaOmega2}, we  have
\begin{equation}\label{est:X}|(X)| 
 \leq C(\mathbb{R}, \nabla \mathbb{R})\tau \|\ww\|_{L^2(0,T; H^1(\Omega_0))}\|\sigma_s\|_{L^2(0,T;\Gamma_{\mathrm{FSI}}^0 )}. 
\end{equation}

The inlet and outlet boundary integrals
\[(XI):=\int_0^{T-\tau}\int_s^{s+\tau} \int_{\Gamma_{\mathrm{in}/\mathrm{out}}} \frac{p_{\mathrm{in}/\mathrm{out}}}{\rho}\mathbb{R}^T\vec{\nn}\cdot \psi dS(y)dtds,\] 
can be estimated again by the Steklov-average property \eqref{stek1} for  \(\varphi(t):=\|p_{\mathrm{in}/\mathrm{out}}(t)\|_{L^2(\Gamma_{\mathrm{in/out}})} \)  and the boundary imbeddings \eqref{n=2,deltaOmega2}, \eqref{n=3,deltaOmega2} 
\begin{equation}\label{est:XI}\begin{aligned}&|(XI)|
 \leq C(\mathbb{R}) \tau \int_0^{T-\tau} \frac{1}{\tau}\int_s^{s+\tau}\|p_{\mathrm{in}/\mathrm{out}}(t)\|_{L^2(\Gamma_{\mathrm{in}/\mathrm{out}})} \;\|\psi(t)\|_{L^2(\Gamma_{\mathrm{in}/\mathrm{out}})} dtds\\
& \leq C(\mathbb{R},\nabla \mathbb{R}) \tau \int_0^{T-\tau}\left(\|\uu(s+\tau)\|_{L^2(\Gamma_{\mathrm{in}/\mathrm{out}})}+\|\uu(s)\|_{L^2}\right)
\left[\|p_{\mathrm{in}/\mathrm{out}}\|_{L^2(\Gamma_{\mathrm{in}/\mathrm{out}})} \right]_\tau (s) ds\\
& \leq C(\mathbb{R},\mathbb{R}^{-1},\nabla \mathbb{R}, \nabla \mathbb{R}^{-1},\Omega_0)\tau \|\uu\|_{L^2(0,T; H^1(\Omega_0))}\;\|p_{\mathrm{in}/\mathrm{out}}\|_{L^2(0,T; L^2(\Gamma_{\mathrm{in}/\mathrm{out}}))}.
\end{aligned}\end{equation}
The estimates \eqref{est:III}, \eqref{est:IV},  \eqref{est:I+II}, \eqref{est:V}, 
 \eqref{est:6n=2}--\eqref{est:XI} of the terms \((I)\)--\((XI)\) arising in the weak formulation \eqref{A.2} conclude the proof of estimate  \eqref{secondestimate} in the case \(n=2\) and \(n=3\). 
\end{proof}

Based on the a-priori estimates \eqref{Aestimes_first_u}, \eqref{Aestimes_first_q} and a compactness argument i.e. the equicontinuity estimate from Theorem \ref{Mainthm} we can now more closely  investigate the limits of the sequences \(\uu_\varepsilon \) and \(\sqrt{\varepsilon}q_\varepsilon\).
\section{Limiting process with $\varepsilon \to 0$}
\label{limit}

In this section we argue
 that the limits \(\varepsilon \to 0\) of \(\uu_\varepsilon\) and \(\sqrt{\varepsilon}q_\varepsilon\) are related to the weak solution of the incompressible Navier-Stokes equations giving a more precise formulation of Corollary \ref{thm:Main2}.
To this end,  we need to have a closer look at the convergence of the sequence \(\uu_\varepsilon\) and the terms in \eqref{weakall} individually.
Strong convergence is needed to treat the terms, where the approximation enters non-linearly, while weak convergence suffices to treat the linear terms.
First, we state the following strong convergence result, which is a consequence of Theorem \ref{Mainthm}.

\begin{lemma}\label{weakstrong}
Under the assumptions of Theorem \ref{Mainthm}, for \(1\leq p<4\),
\begin{equation}\label{comp}\uu_\varepsilon \rightarrow \uu_0\text{ strongly in }L^p((0,T) \times \Omega_0) \text{ as } \varepsilon \to 0 \end{equation}
\begin{proof}
    The first a-priori estimate \eqref{Aestimes_first_u}
 grants that \(\uu_\varepsilon\)  has subsequences that converge weakly in \(L^2(0,T;H^1(\Omega_0))\) as \(\varepsilon\to 0\) based on the reflexivity for Hilbert-spaces \cite{Alt}.
 
 The second a-priori estimate  \eqref{secondestimate} together with the uniform boundedness of \(\uu_\varepsilon\) due to  \eqref{Aestimes_first_u} 
grants the convergence \(\sqrt{J}\uu_\varepsilon\rightarrow \sqrt{J}\uu_0\)  {strongly} in \(L^1((0,T)\times\Omega_0)\) based on the Alt-Luckhaus compactness Argument, see \cite[Lemma 1.9]{luckhaus} with \(b(\uu_\varepsilon)=\uu_\varepsilon\) and \(B(\uu_\varepsilon)=\frac{1}{2}|\uu_\varepsilon|^2\). 
By assumption \eqref{boundd2} \(\sqrt{J}\) is bounded, boundedly invertible in \(L^\infty(0,T\times\Omega_0) \), thus \(\uu_\varepsilon\to \uu_0\) in \(L^1((0,T)\times \Omega_0)\) strongly. 

By the uniform \(L^4((0,T)\times \Omega_0)\)-boundedness of \(\uu_\varepsilon\)   (for \(n=3\) as an assumption, for \(n=2\) as a consequence of the Sobolev-imbedding \eqref{n=2,Omega0}) 
 the strong \(L^1((0,T)\times \Omega_0)\) convergence can be extrapolated to get the strong convergence  in \(L^p((0,T)\times \Omega_0)\) for \(1< p<4\). Namely, for \(\frac{1}{p}=\theta+\frac{1-\theta}{4}\) and some \(0<\theta<1\), one has
\[\begin{aligned}\|\uu_\varepsilon-\uu_0\|_{L^p((0,T)\times \Omega_0)} &\leq \|\uu_\varepsilon-\uu_0\|_{L^1((0,T)\times \Omega_0)}^\theta \; \|\uu_\varepsilon-\uu_0\|_{L^4((0,T)\times \Omega_0)}^{1-\theta}\\ &\leq C^{1-\theta} \|\uu_\varepsilon-\uu_0\|_{L^1((0,T)\times \Omega_0)}^\theta, \end{aligned}\] see  \cite[Appendix B.2.h]{Evans}. Therefore  \(\|\uu_\varepsilon-\uu_0\|_{L^p((0,T)\times \Omega_0)} \to 0\) for \(\varepsilon \to 0, 1\leq p<4\). 
\end{proof}
\end{lemma}

Note that from the a-priori estimate \eqref{Aestimes_first_q}, we cannot grant that the pressure \(q_\varepsilon\) itself has a weakly convergent subsequence. 
However \(\varepsilon q_\varepsilon\) converges  weakly in \(L^2(0,T;H^1(\Omega_0))\) to \(0\)  due to the uniform boundedness of \(\sqrt{\varepsilon} q_\varepsilon\) and the   additional factor \(\sqrt{\varepsilon}\).

From the first a-priori estimate \eqref{Aestimes_first_u}, it also follows that \(\nabla\uu_\varepsilon\)  has subsequences that converge weakly to \(\nabla\uu_0\) in \(L^2(0,T;L^2(\Omega_0))\) as \(\varepsilon\to 0\) (see \cite[Beispiel 6.4(3)]{Alt}).

With the strong convergence from Lemma \ref{weakstrong}, we can now show that the limiting process \(\varepsilon \to 0\) in the artificial compressibility formulation \eqref{weakall} leads to an incompressible formulation of the fluid-problem on a deforming domain. 

\begin{theorem} Under the assumptions of Theorem \ref{Mainthm}, the limit \(\uu_0\) of \(\uu_\varepsilon\) is divergence-free almost everywhere and correspondingly, the fluid motion equations \eqref{first equation}, \eqref{presure eq}  with \(\varepsilon\) replaced by \(0\), i.e.
\begin{equation*}\partial_t \vv_0 +(\vv_0\cdot \nabla )\vv_0-2\frac{\mu}{\rho}\div[e(\vv_0)]+\frac{1}{\rho}\nabla p_0=\vec 0,\quad \div \vv_0=0\end{equation*} in the deforming domain \(\Omega_t\) with boundary conditions \eqref{eq:w}, \eqref{tensor}, \eqref{kinematicp}, \eqref{vwall}  and zero initial condition hold in the weak sense. 
\end{theorem}
\begin{proof}

 \textit{(I) Divergence-freeness of the weak solution:}\\
First we show that the weak limit \(\uu_0=\lim_{\varepsilon \to 0} \uu_\varepsilon\) is divergence-free almost everywhere, and thus the fluid is incompressible in the limit in the weak sense. 
To this end, we proceed testing \eqref{weakall} with \(\psi=0\) and arbitrary \(\Phi\) to get
\[0=\frac{\varepsilon}{\rho^2}\int_0^{T}\int_{\Omega_0}J^{-1}(\mathbb{R}^{T}\nabla q_\varepsilon)\cdot (\mathbb{R}^{T}\nabla \Phi) dydt+\frac{1}{\rho}\int_0^{T}\int_{\Omega_0}\div(\mathbb{R} \uu_\varepsilon)\Phi dydt.\]
Rearranging the terms and double usage of the Cauchy-Schwarz inequality yields \begin{equation} \label{eq:free}\begin{aligned}&\left|\int_0^{T}\int_{\Omega_0}\div(\mathbb{R} \uu_\varepsilon)\Phi dydt\right|\leq \frac{\sqrt{\varepsilon}}{\rho}\int_0^{T}\int_{\Omega_0} J^{-1}|\mathbb{R}^{T}\nabla (\sqrt{\varepsilon}q_\varepsilon)|\; |\mathbb{R}^{T}\nabla \Phi| dydt
\\& \leq \sqrt{\varepsilon} \, C(\mathbb{R})\|\sqrt{\varepsilon} q\|_{L^2(0,T;H^1(\Omega_0))}\; \|\Phi\|_{L^2(0,T;H^1(\Omega_0))} .
\end{aligned}\end{equation}
Since the test function is arbitrary in \(L^2(H^1(\Omega_0))\) and 
 $\|\sqrt{\varepsilon} q_{\varepsilon}\|_{L^2(0,T;H^1(\Omega_0))}$ is bounded uniformly on $\varepsilon$, cf. \eqref{Aestimes_first_q}, and
the constant $C(\mathbb{R})$ does not depend on \(\varepsilon\), it follows that 
$\div(\mathbb{R}\uu_\varepsilon) \rightarrow 0$  as \( \varepsilon \rightarrow 0\) almost everywhere
 in $(0,T)\times\Omega_0$ and the (weak) limit   $\mathbb{R}\uu_0$ is divergence-free a. e.  By the transformation of integrals (see \eqref{eq:inttransformation}, \eqref{eq:divtransform}) \[\begin{aligned}\int_{\Omega_t} \div_x \vv_\varepsilon\;  \phi dx&=\int_{\Omega_0}JJ^{-1}\div_y (\mathbb{R}\uu_\varepsilon) \Phi dy=\int_{\Omega_0}\div_y (\mathbb{R}\uu_\varepsilon) \Phi dy
\end{aligned},\] with testfunctions \(\phi\) defined in \(\Omega_t\) and the corresponding \(\Phi\) in \(\Omega_0\), it follows that the limit \(\vv_0\) also is divergence-free on \(\Omega_t\).
As a consequence, the testfunctions \(\psi\), cf. \eqref{def:psi}, \eqref{presure eq}, can also be considered as divergence-free on \(\Omega_0\) in the sense of its Piola transformation \(\mathbb{R}\psi\).\\
\medskip\\
\textit{(II) Limiting process in the weak formulation:}\\ 
We now let \(\varepsilon \to 0\) in all terms of the identity \eqref{weakall} considering divergence-free testfunctions, i.e. \(\div_y(\mathbb{R}\psi)=0\) a. e.

\textit{(1) Limit of the nonlinear convective terms:}\\
The weak formulation \eqref{weakall} contains two non-linear convective terms that converge by similar argumentation. First, we consider \(\psi\in C((0,T)\times \overline{\Omega_0)}\) and show that \[\int_0^T \int_{\Omega_0} (\nabla_y \uu_\varepsilon)^T \mathbb{R} \uu_\varepsilon \cdot \psi dydt\to \int_0^T \int_{\Omega_0} (\nabla_y \uu_0)^T \mathbb{R} \uu_0\cdot \psi dydt \text{ as } \varepsilon \to 0.\] 
To this end, note that \(\mathbb{R}\) is bounded in \(L^\infty((0,T)\times \Omega_0)\) by assumption \eqref{ass:globdiff} and
\begin{equation}\label{convlimit} \begin{aligned}
&\int_0^T\int_{\Omega_0}\big[(\nabla_y \uu_\varepsilon)^T\mathbb{R}\uu_\varepsilon-(\nabla_y\uu_0)^T \mathbb{R} \uu_0 \big]\cdot \psi dydt \\
&= \int_0^T\int_{\Omega_0}(\nabla_y \uu_\varepsilon)^T  \mathbb{R}(\uu_\varepsilon-\uu_0)\cdot \psi dydt+\int_0^T\int_{\Omega_0} \mathbb{R}\uu_0 \cdot (\nabla_y \uu_\varepsilon-\nabla_y\uu_0)\psi dydt.
\end{aligned}\end{equation}
The first  summand  can be estimated by \[C(\mathbb{R})\|\psi\|_{C((0,T)\times \overline{\Omega_0})} \|\nabla_y \uu_\varepsilon\|_{L^2(0,T;L^2(\Omega_0))} \|\uu_\varepsilon -\uu_0\|_{L^2(0,T;L^2(\Omega_0))}\]
and thus converges to zero by the strong convergence \eqref{comp} of \(\uu_\varepsilon \to \uu_0\) in \\ \(L^2(0,T;L^2( \Omega_0))\)
and the uniform boundedness of \(\nabla_y \uu_\varepsilon\) in \(L^2(0,T;L^2(\Omega_0))\), \eqref{Aestimes_first_u}. \\The second summand converges to zero by  the weak convergence of \(\nabla \uu_\varepsilon \psi\rightharpoonup  \nabla \uu_0\psi\) in \(L^2(0,T;L^2(\Omega_0))\) and \(\mathbb{R}\uu_0\in L^2(0,T;L^2(\Omega_0))\). 
Second, we show that the second part of the convective term \begin{equation}\label{eq:div}\frac{1}{2}\int_0^T \int_{\Omega_0}\div(\mathbb{R}\uu_\varepsilon)\uu_\varepsilon \cdot \psi dydt\to 0 \text{ as } \varepsilon \to 0\end{equation} for continuous testfunctions \(\psi\in C((0,T)\times \Omega_0)\).
To this end, note that \(\mathbb{R}\) is bounded in \(W^{1,\infty}((0,T)\times\Omega_0)\) by assumption \eqref{defW} and
\begin{equation}\label{convlimit2} \begin{aligned}
&\int_0^T\int_{\Omega_0}\div(\mathbb{R}\uu_\varepsilon)\uu_\varepsilon \cdot \psi dydt \\
&= \int_0^T\int_{\Omega_0}\div(\mathbb{R}\uu_\varepsilon)(\uu_\varepsilon-\uu_0) \cdot \psi dydt + \int_0^T\int_{\Omega_0}\div(\mathbb{R}\uu_\varepsilon)\uu_0\cdot\psi  dydt, 
\end{aligned}\end{equation}
where similarly  \(\uu_\varepsilon \to \uu_0\) strongly and \(\div(\mathbb{R}\uu_\varepsilon)\) is uniformly bounded in  \(L^2(0,T;L^2(\Omega_0))\) for the first summand. For the second summand   \(\div(\mathbb{R} \uu_\varepsilon)\rightharpoonup 0\)  weakly (see \eqref{eq:free} and lines below) and \(\uu_0\cdot \psi\in L^2(0,T;L^2(\Omega_0))\). For non-continuous testfunctions \(\psi\) as in \eqref{assu}, the convergence follows by a standard density argument for Sobolev functions. 

Here, we have demonstrated that the additional nonlinear term \eqref{eq:div}, which was added to the convective term, cf. \eqref{def:b}, to compensate for the loss of selenoidality of the weak solution, disappears in the limit. This confirms the consistency of our approximation with the original incompressible problem, where this term is not present.

\textit{(2) Convergence of the  nonlinear inlet/outlet terms:}\\
The weak \(L^2(0,T;H^1(\Omega_0))\)-convergence  of the space \(\uu_\varepsilon\) together with the boundary imbedding of \(H^1(\Omega_0)\) in \(L^2(\partial \Omega_0)\) can be used for the convergence of the non-linear inlet/outlet terms.
 One has  
\begin{equation}\begin{aligned}
&\int_0^T \int_{\Gamma_{\mathrm{in}/\mathrm{out}}}
\big((\uu_\varepsilon\cdot \psi)\mathbb{R}\uu_\varepsilon -(\uu_0\cdot \psi)\mathbb{R}\uu_0\big)\cdot\vec{\nn}dS(y)dt
\\ &=     \int_0^T\int_{\Gamma_{\mathrm{in}/\mathrm{out}}}
\big((\uu_\varepsilon\cdot \psi)\mathbb{R}(\uu_\varepsilon -\uu_0\big)\cdot\vec{\nn} + 
\big((\uu_\varepsilon-\uu_0)\cdot \psi)\mathbb{R}\uu_0)\cdot \vec{\nn}dS(y)dt.
\end{aligned}\end{equation}
We proceed  considering at first testfunctions \(\psi\in C((0,T)\times \overline{\Omega_0})\) continuous up to the boundary and  \(\mathbb{R}\) bounded by \eqref{boundd2}.
In the case of \(n=2\), one can estimate the first summand using   \eqref{n=2,deltaOmega2} by  \[\begin{aligned}&\left|\int_0^T\int_{\Gamma_{\mathrm{in}/\mathrm{out}}}
(\uu_\varepsilon\cdot \psi)\mathbb{R}(\uu_\varepsilon -\uu_0\big)\cdot\vec{\nn}dS(y)dt\right|\\ & \leq C(\mathbb{R},\|\psi\|)\int_0^T \|\uu_\varepsilon\|_{L^2(\partial \Omega_0)} \; \|\uu_\varepsilon-\uu_0\|_{L^2(\partial \Omega_0)}dt\\
& \leq C(\mathbb{R},\Omega_0,\|\psi\|)\int_0^T \|\uu_\varepsilon\|_{L^2(\Omega_0)} \;\|\uu_\varepsilon-\uu_0\|_{L^2(\Omega_0)} dt\\ & \leq C(\mathbb{R},\Omega_0,\|\psi\|)  \|\uu_\varepsilon\|_{L^2(0,T;L^2(\Omega_0))}\; \|\uu_\varepsilon-\uu_0\|_{L^2(0,T;L^2(\Omega_0))},\end{aligned}\]
where \(\|\psi\|=\|\psi\|_{C((0,T)\times \overline{\Omega_0})}\), the first factor is uniformly bounded by the a-priori estimate \eqref{Aestimes_first_u} and the second factor converges to zero by the strong convergence  \eqref{comp}. The convergence of the second summand  to zero follows similarly with 
\[\begin{aligned}&\left|\int_0^T\int_{\Gamma_{\mathrm{in}/\mathrm{out}}}
(((\uu_\varepsilon-\uu_0)\cdot \psi)\mathbb{R}\uu_0)\cdot \vec{\nn}dS(y)dt\right|\\ &\leq C(\mathbb{R},\Omega_0, \|\psi\|)\|\uu_0\|_{L^2(0,T;L^2(\Omega_0))}\; \|\uu_\varepsilon-\uu_0\|_{L^2(0,T;L^2(\Omega_0))}.\end{aligned}\]
In the case of \(n=3\) the estimates for the first summand have to be modified accordingly using \eqref{n=3,deltaOmega2} 
\[\begin{aligned} &C(\mathbb{R},\|\psi\|) \int_0^T\|\uu_\varepsilon\|_{L^2(\partial \Omega_0)} \; \|\uu_\varepsilon-\uu_0\|_{L^2(\partial \Omega_0)} dS(y)dt\\ &\leq C(\mathbb{R},\Omega_0,\|\psi\|)\int_0^T \|\uu_\varepsilon\|_{H^1(\Omega_0)}\|\uu_\varepsilon-\uu_0\|_{H^1(\Omega_0)}^{1/2}\; \|\uu_\varepsilon-\uu_0\|_{L^2(\Omega_0)}^{1/2}dt \\ 
& \leq C(\mathbb{R},\Omega_0,\|\psi\|)\|\uu_\varepsilon\|_{L^2(0,T;H^1(\Omega_0))}\\ & \qquad \cdot\left(\|\uu_\varepsilon\|_{L^2(0,T;H^1(\Omega_0))}+\|\uu_0\|_{L^2(0,T;H^1(\Omega_0))}\right)^{1/2}\;\|\uu_\varepsilon-\uu_0\|_{L^2(0,T;L^2(\Omega_0))}^{1/2}.
\end{aligned}\]
The first factors are uniformly bounded by \eqref{Aestimes_first_u} and the last factor converges to zero by  \eqref{comp}.
By a standard density argument for Sobolev functions the convergence to zero follows for non-continuous testfunctions \(\psi\) as in \eqref{assu} as well.

\textit{(3) Convergence of the remaining linear terms:}\\
The remaining  volume and boundary terms in \eqref{weakall} that contain \(\uu_\varepsilon\) only linearly  converge using the corresponding weak-\(L^2(0,T;H^1(\Omega_0))\) convergence result for \(\uu_\varepsilon, \sqrt{\varepsilon}q_\varepsilon\). The limiting argumentation is straightforward and omitted here. 

\textit{(4) The case $\varepsilon=0$:}\\
After letting \(\varepsilon \to 0\) in all terms of \eqref{weakall} and by considering divergence-free testfunctions, the weak formulation variant of \eqref{weakall} reads for \(\varepsilon=0\) as 

\begin{eqnarray}
\label{weakall=0}
0&=-&\hspace{-10pt}\int_0^{T}\int_{\Omega_0}J\uu_0 \cdot \frac{\partial}{\partial t}\psi dy dt -\int_0^{T}\int_{\Omega_0}(\partial_t J) \uu_0 \cdot \psi dydt
-\int_0^T\int_{\Omega_0}(\nabla_y \uu_0)^T\mathbb{R}\ww\cdot \psi dydt\nonumber
\\ &&+\int_0^{T}\Bigg[\hat b(\uu_0,\uu_0,\psi)-\frac{\mu}{\rho}\int_{\Omega_0}J^{-1}[\mathbb{R}^T\nabla_y\uu_0+(\mathbb{R}^T\nabla_y\uu_0)^T]:[\mathbb{R}^T\nabla_y\psi]dy\\
&&+\frac{1}{\rho}\int_{\Gamma_{\mathrm{FSI}}^0}\sigma_s\vec{\nn}\cdot \psi dS(y)+\int_{\Gamma_{\mathrm{in}/\mathrm{out}}} \left(\frac{p_{\mathrm{in}/\mathrm{out}}}{\rho}\mathbb{R}^T\vec{\nn}\cdot \psi-\frac{1}{2}(\uu_0\cdot \psi)\mathbb{R}\uu \cdot  \vec{\nn}\right) dS(y)\Bigg],\nonumber\end{eqnarray} 
where
\[\uu_0|_{\Gamma_{\mathrm{wall}}}=0,\qquad \uu_0\times \vec{\nn}|_{\Gamma_{\mathrm{in}/\mathrm{out}}}=0 ,\qquad \uu_0|_{\Gamma_{\mathrm{FSI}}^0}=\ww \qquad \] satisfying the boundary conditions \eqref{eq:w}, \eqref{tensor},  \eqref{kinematicp}, \eqref{vwall}  
with the initial condition at \(t=0\) \(\uu_0(y,t=0)=0\)  for all  \(y \in \Omega_0\) 
 for all testfunctions 
  \[\psi\in 
H^1(0,T;H^1(\Omega_0)), \quad \psi|_{\Gamma_{\mathrm{wall}}}=0, \quad \div(\mathbb{R}\psi)=0 \text{ a. e.}, \quad \psi(T)=0 .\]
Note that here \(\frac{1}{\rho}\int_{\Omega_0}\div_y(\mathbb{R}\uu_0)\Phi dydt\) is not present in \eqref{weakall=0} due to the divergence-freeness of the limit \(\mathbb{R}\uu_0\), see point (I). Moreover, the (transformed) pressure \(q\) is (as usual) not present in the weak formulation limit \eqref{weakall=0} due to the solenoidality of the testfunctions \((\mathbb{R}\psi)\).
Lastly, 
\(\hat b(\uu_0,\uu_0,\psi)\) simplified to \(\int_{\Omega_0}(\nabla_y\uu_0)^T \mathbb{R}\uu_0\cdot \psi dy \) in the weak formulation after the limiting process \(\varepsilon \to 0\) for vanishing artificial compressibility in \eqref{weakall=0}, see point (I).

We conclude that by \eqref{weakall=0} the transformed to \(\Omega_0\) weak formulation for divergence-free velocities has been obtained, which corresponds to the incompressible Navier-Stokes equations 
\begin{equation*}\partial_t \vv +(\vv\cdot \nabla )\vv-2\frac{\mu}{\rho}\div[e(\vv)]+ \frac{1}{\rho}\nabla p=\vec 0,\quad \div( \vv)=0 \text{ in } \Omega_t\end{equation*}
 where the velocity is divergence-free and the \(\frac{1}{2}\vv \div{\vv}\) term is not present (see \cite[Section III.8]{Temam}).

\end{proof}

\section{Conclusion}\label{sec:conclusion}

 In this work, we introduced a small artificial incompressibility with a parameter \(\varepsilon>0\) in our approximate system \eqref{first equation}, \eqref{presure eq}. 
This  approximation serves to overcome the difficulties associated with solenoidal functional spaces for fluid-flow problems on moving domains.
Under mild regularity conditions on the given domain deformation 
\(\dd\), we  show that artificial compressibility approximations \(\uu_\varepsilon\) converge to divergence-free weak solutions of the Navier-Stokes fluid flow problem based on an alternative compactness argument via the equicontinuity estimate in time, see Theorem \ref{Mainthm}. Note that strong convergence cannot be  straightforwardly obtained by the classical Lions-Aubins Lemma since the estimate on the time derivative depends on \(\varepsilon\). The regularity assumptions on the domain deformation  here are related to a bijective  map of the fluid domain to a fixed reference domain. In case \(n=3\), additional uniform \(L^4(0,T; L^4(\Omega_0))\)- boundedness  of the velocity is imposed to obtain the equicontinuity result. 

 The main result stated in Theorem  \ref{Mainthm} poses a compactness argument, which has  been proven here for a deforming, time-dependent domain  by construction of  special test functions including the difference of two Piola-transformed  weak solutions on different time points, preserving the divergence operator for different domain time-points. This result can be generalized for full fluid-structure interaction problems, where a domain-deformation equation is additionally coupled. However, the 'natural' regularity of the domain deformation  given by usual structure models is not sufficient to obtain compactness results, thus further approximating or regularizing steps are necessary and have to be investigated for full fluid-structure interaction problems.
 
\section*{Ackgnowledgments}

The authors would like to thank J\'{a}n Filo (Comenius University, Bratislava) for fruitful discussions during visits in Landau and Bratislava, respectively.


\newpage
\section{Appendix} 
\label{Appendix}

The purpose of this appendix is to give an overview on the (standard) calculus applied in the derivation of the presented estimates and proofs.\\
\textbf{Transformation of  volume-integrals:} For a scalar function \(f\) and a coordinate transformation \(x=A_t(y)\) it holds
\begin{equation}\label{eq:inttransformation}\int_{\Omega_t}f(x,t)dx=\int_{\Omega_0}f(A_t(y))|\det \nabla_yA_t(y)|dy=\int_{\Omega_0}f(A_t(y))J(y)dy.\end{equation} 
\textbf{Transformation of the divergence operator} (cf.  \cite{anna2}): 
The transformation of the  divergence operator between \(\Omega_t\) and the reference domain \(\Omega_0\) reads as follows:
\begin{equation}\label{eq:divtransform}\div_x \vv=J^{-1}\div_y( \mathbb{R}\uu ),\end{equation}
where we used the abbreviations \( \mathbb{J}:=(\nabla_y A_t)^T, \mathbb{R}:=J\mathbb{J}^{-1}= \mathrm{cof}\nabla A_t\) as a cofactor. For verification, note that by the product rule 
 \[\div_y(\mathbb{R}\uu)=(\div_y \mathbb{R})\cdot \uu+(\mathbb{R}^T\nabla_y)\cdot \uu=(\mathbb{R}^T\nabla_y)\cdot \uu\] 
using the Piola identity \(\div (\mathrm{cof} \nabla A_t)=0\).\\
\textbf{Transformation of gradients}: For the scalar valued functions \(q(y,t)=p(x,t)=p(A_t(y),t)\) the chain-rule for gradients  yields  \[\nabla_y q(y,t)=(\nabla_y A_t)(\nabla_x p(x,t))=\mathbb{J}^T \nabla_xp(x,t)=J\mathbb{R}^{-T}\nabla_x p(x,t).\]  
 Using \(\uu(y,t)=\vv(A_t(y),t)\)  and the chain-rule for gradient matrices, it holds
\begin{equation}\label{eq:gradtransform}\nabla_y \uu(y,t)=\nabla_y A_t(y)\nabla_x \vv(x,t) =\mathbb{J}^{T}\nabla_x \vv(x,t)=J\mathbb{R}^{-T}\nabla_x\vv(x,t).\end{equation} 
\textbf{Transformation of boundary integrals}:
  Let \(\sigma_x\) be a tensor on \(\partial\Omega_t\) and \(\varphi\) be a vector-valued testfunction in \(\Omega_t\) and  \(x=A_t(y)\), \(y\in \Omega_0\). Then we set the transformed tensor as \(\sigma_y(y):=\sigma_x(A_t(y))\) on \(\partial \Omega_0\) and the corresponding testfunction as \(\psi\).
  By the divergence theorem, \eqref{eq:inttransformation} and \eqref{eq:divtransform}, one then has the transformation of the boundary integral
\begin{equation} \label{eq:tensortransform}\begin{aligned}\int_{\partial \Omega_t}\sigma_x \vec{\nn}\cdot \varphi dS(x)&=\int_{\partial \Omega_t}(\sigma_x^T \varphi) \cdot \vec{\nn}  dS=\int_{\Omega_t}\div_x(\sigma_x^T\varphi)dx\\
&=\int_{\Omega_0}\div_y(\mathbb{R}(\sigma_y^T\psi))dy
=\int_{\partial\Omega_0}\mathbb{R}(\sigma_y)^T\psi\cdot \vec{\nn} dS(y)\\&=\int_{\partial \Omega_0}\sigma_y\mathbb{R}^T \vec{\nn}\cdot \psi dS(y),
\end{aligned}.\end{equation}
This calculus (in combination with \eqref{eq:gradtransform}) has been used for the fluid Cauchy stress tensor acting on the interface \(\Gamma_{\mathrm{FSI}}\) to derive the factor \(\mathbb{R}^T\) in the transition from the  deforming domain \eqref{tensorf} to the fixed reference domain  \eqref{tensor0}.\\
\textbf{Time-derivative:}
The ALE-Transport-Theorem and the Stokes Equation together imply the following identities: \begin{equation}\label{ALE}\begin{aligned}\frac{d}{dt}\int_{\Omega_t}f dx
&= \int_{\Omega_t}D^A_tf+f\div \ww dx
:=\int_{\Omega_t}\frac{\partial}{\partial t}f+(\nabla_xf)^T  \ww+ f\div \ww dx\\ &=\int_{\Omega_t}\frac{\partial}{\partial t}f+\div(f \ww) dx= \int_{\Omega_t}\frac{\partial}{\partial t}f dx+\int_{\partial \Omega_t}f\ww\cdot \vec{\nn} dS(x),\end{aligned}\end{equation}
where \(D_f^Af:=\frac{\partial}{\partial t}f+ (\nabla_xf)^T\ww\).\\
\textbf{Generalized Poincar\'e-inequality} \cite[Corollary 4.5.2]{Ziemer}:  For \(u\in W^{1,p}(\Omega_0)\)  with zero boundary values on \(N=\{x: u(x)=0\}\), one has that 
\begin{equation}\label{GenPoinc}\|u\|_{L^p(\Omega_0)}\leq C (B_{1,p}(N))^{-1/p}\|\nabla u\|_{L^p(\Omega_0)},\end{equation} where \(\Omega\) is a bounded extension domain (which can be granted by a Lipschitz boundary, see \cite[Section 2.5, p.64]{Ziemer}), and \(B_{1,p}\) is the Bessel capacity. 

\textbf{Steklov-average:}
We  define the Steklov-average of  \(\varphi\) as \begin{equation}\label{stek0}[\varphi]_\tau(s):=\frac{1}{\tau} \int_s^{s+\tau}\varphi(t)dt.\end{equation} 

The Steklov-average has the well-known property for \(1\leq p<\infty\) (see \cite{steklov}, \cite{anna2})
\begin{equation}\label{stek1}\| [\varphi]_\tau\|_{L^p(0,T-\tau)}\leq \|\varphi\|_{L^p(0,T)}.\end{equation}

\end{document}